\documentclass[final]{siamltex}

\usepackage{amsmath}
\usepackage{epsfig}
\usepackage{amsfonts}
\usepackage{latexsym}
\usepackage{subfig}
\usepackage{bbm}

\def\Int{\mathcal{I}}

\def\Ordo{\mathcal{O}}
\def\Real{\mathbb{R}}  
\def\Cnumbers{\mathbb{C}}  
\def\Znumbers{\mathbb{Z}}  

\newcommand{\lbeq}[1]{{\label{OR:eq:#1}}}

\newcommand{\lbsec}[1]{{\label{OR:sec:#1}}}

\newcommand{\lbtheo}[1]{{\label{OR:theo:#1}}}
\newcommand{\lblem}[1]{{\label{OR:lem:#1}}}

\newcommand{\lbrem}[1]{{\label{OR:rem:#1}}}

\newcommand{\eq}[1]{{(\ref{OR:eq:#1})}}
\newcommand{\eqtwo}[2]{{(\ref{OR:eq:#1}, \ref{OR:eq:#2})}}

\newcommand{\sect}[1]{{Section~\ref{OR:sec:#1}}}

\newcommand{\appen}[1]{{Appendix~\ref{OR:sec:#1}}}
\newcommand{\theo}[1]{{Theorem~\ref{OR:theo:#1}}}
\newcommand{\lem}[1]{{Lemma~\ref{OR:lem:#1}}}

\newcommand{\rem}[1]{{Remark~\ref{OR:rem:#1}}}

\newcommand{\be}[1]{\begin{equation} \lbeq{#1}}
\newcommand{\ee}{\end{equation}}
\newcommand{\bea}[1]{\begin{eqnarray} \lbeq{#1}}
\newcommand{\eea}{\end{eqnarray}}
\newcommand{\bear}{\begin{array}}
\newcommand{\eear}{\end{array}}

\newtheorem{remark}{Remark}[section]

\usepackage{color}

\title{Sobolev and Max Norm Error Estimates for Gaussian Beam Superpositions\\ \today}

\author{
Hailiang Liu\thanks{Department of Mathematics, Iowa State University, Ames, IA 50010, USA}
 \and 
 Olof Runborg\thanks{Department of Mathematics and Swedish e-Science Research Center (SeRC), KTH, 10044 Stockholm, Sweden}
 \and 
 Nicolay~M.~Tanushev\thanks{Z-Terra Inc., 
17171 Park Row, Suite 247,
Houston TX 77084, USA}
 }

\date{\today}

\begin{document}
\maketitle

\begin{abstract}
This work is concerned with the
accuracy of Gaussian beam superpositions, which
are asymptotically valid high frequency solutions to linear hyperbolic
partial differential equations and the Schr\"odinger equation.
We derive Sobolev and max norms estimates for the difference
between an exact solution and the corresponding Gaussian beam approximation,
in terms of the short wavelength~$\varepsilon$.
The estimates are performed for
the scalar wave equation and the
Schr\"odinger equation.
Our result demonstrates that a Gaussian beam
superposition with $k$-th order beams converges to the exact solution
as $O(\varepsilon^{k/2-s})$ in order $s$ Sobolev norms. 
This result is valid in any number of spatial
dimensions and it is unaffected by the presence of caustics in the solution. 
In max norm, we
show that away from caustics the convergence rate is $O(\varepsilon^{\lceil k/2\rceil})$ and away from the essential support of the solution, the 
convergence is spectral in $\varepsilon$. However, in the neighborhood of a caustic point
we are only able to show the slower, and dimensional dependent, rate
$O(\varepsilon^{(k-n)/2})$ in $n$ spatial dimensions.
\end{abstract}



\pagestyle{myheadings}
\thispagestyle{plain}
\markboth{H. LIU, O. RUNBORG AND N. M. TANUSHEV}{ERROR ESTIMATES FOR GAUSSIAN BEAM SUPERPOSITIONS}

\section{Introduction}

In this paper we consider the accuracy of
Gaussian beam approximations for two time-dependent
partial differential equations (PDEs) 
with highly oscillatory solutions:
the dispersive Schr\"odinger equation
in the semi-classical regime,
\begin{align}\lbeq{Schrodinger}
  -i\varepsilon u_t - \frac{\varepsilon^2}{2}\triangle u + V(y)u &= 0, \qquad (t,y)\in (0,T]\times \Real^n, \\
  u(0,y) &=B_0(y)e^{i\varphi_0(y)/\varepsilon} \notag \ ,
\end{align}
and the scalar wave equation,
\begin{align}\lbeq{WaveEq}
u_{tt} - c(y)^2\Delta u &= 0, \qquad (t,y) \in (0,T]\times\Real^n, \\ 
 u(0,y) &= B_0(y)e^{i\varphi_0(y)/\varepsilon}, \notag \\ 
 u_t(0,y) &= \varepsilon^{-1}B_1(y)e^{i\varphi_0(y)/\varepsilon}.  \notag
\end{align}
In these equations,
$V(y)$ is an external potential,
$c(y)$ is the speed of propagation
and $\varepsilon\ll 1$
is the short wavelength, or the scaled Planck constant for \eq{Schrodinger}.
Since $\varepsilon$ is small, the initial data for both PDEs
are highly oscillatory. The amplitude functions 
$B_{{\ell}}$ and phase $\varphi_0$
are real valued functions on $\Real^n$.
We will assume that $c,V,\varphi_0,B_\ell$ are all smooth
and
that $B_{{\ell}}$ are supported in
the compact set $K_0\subset\Real^n$.

Direct numerical
simulation of these PDEs is expensive when $\varepsilon$ is small.
A large number of grid points is needed to resolve the wave
oscillations and the computational cost to maintain constant accuracy grows
rapdily with the frequency. 
As an alternative
 one can use high frequency asymptotic models for wave propagation, such as geometrical optics \cite{Keller:62, EngRun:03, Runborg:07}, which is obtained 
 in the limit when $\varepsilon\to 0$.
 The solution of the PDE
 is then written as
\begin{align}\lbeq{GOform}
u(t,y) = a(t,y,\varepsilon) e^{i\phi(t,y)/\varepsilon},
\end{align}
where $\phi$ is the phase, and $a$ is the
amplitude of the solution, which both vary on a much coarser scale
than $u$.
When $\varepsilon\to 0$
the phase and amplitude are independent
of the frequency.
 Therefore, they can be computed at a computational cost independent of the frequency. However, at caustics where rays concentrate,
geometrical optics breaks down and the predicted amplitude
becomes unbounded, \cite{Ludwig:66, Kravtsov:64}.

Gaussian beams form another high frequency asymptotic model which is closely
related to geometrical optics
\cite{Ralston:82,Popov:1982,Katchalov_Popov:1981,Cerveny_etal:1982,Klimes:1984,Hormander:71,Hagedorn:80}.
Unlike geometrical optics, there is
no breakdown at caustics.
The solution is 
assumed to be of the same form \eq{GOform}, but a Gaussian beam is
a localized solution that concentrates near a single geometrical
optics ray $x(t)$ in space-time. We write it as
$$
v(t,y)=A(t,y-x(t))e^{i\Phi(t,y-x(t))/\varepsilon}.
$$
The concentration comes from the fact that, 
although the phase function is real-valued along $x(t)$, 
it has a positive {\em imaginary} 
part away from $x(t)$.
Moreover, the imaginary part is quadratic in $y$ so that
$\Im\Phi(t,y)\sim |y|^2>0$, and therefore
$|v(t,y)|\sim e^{-|y-x(t)|^2/\varepsilon}$, which means that
the beams have essentially a Gaussian shape of
width $\sqrt{\varepsilon}$, centered around $x(t)$.
Because of this localization one can approximate the amplitude and phase
away from $x(t)$ by Taylor expansion;
both $\Phi(t,y)$
 and $A(t,y)$ are polynomials in $y$. 
 For instance, in first order beams $\Phi(t,y)$ is a second order polynomial, and
 $A(t,y)$ is a zeroth order (constant) polynomial.
 The coefficients in the
 polynomials satisfy ODEs.
 Higher order Gaussian beams are created by using an asymptotic
series for the amplitude and using higher order Taylor expansions 
for $\Phi(t,y)$
 and $A(t,y)$.
For higher order beams, a cutoff function is also necessary to
avoid spurious growth away from the center ray. 

In numerical methods one must consider
more general high frequency solutions,
which are not necessarily concentrated on a single ray. 
Superpositions of Gaussian beams are then used. This is natural 
since the PDEs are linear.
If we let
$v(t,y,z)$
be a beam starting from the point $y=z$,
the Gaussian beam superposition is defined as
\be{superpos}
  u_{GB}(t,y) = \left(\frac{1}{2\pi\varepsilon}\right)^\frac{n}{2} 
  \int_{K_0} v(t,y,z)dz,
\ee
for the
set $K_0$ where initial data is concentrated.
The prefactor 
normalizes the superposition appropriately, so that
$u_{GB}=O(1)$. More details about the construction of
Gaussian beam superpositions are given in \sect{GBdef}.

Numerical methods based on Gaussian beam type 
superpositions go back to the 1980's 
for the wave equation \cite{Popov:1982,Katchalov_Popov:1981,
Cerveny_etal:1982,Klimes:1984,White_etal:87} and for the
Schr\"odinger equation \cite{Heller:81,HermanKluk:84}.
Since then a great many such methods have
been developed for various applications
\cite{Hill:1990,Hill:2001,TQR:2007,FaouLubich:06,LeungQian:09,
Tanushev:08,
WuHuangJinYin:12,
MotRun:15,
QianYing:2010}.
Typically, the ODEs for the Taylor coefficients of the phase and amplitude
are solved
using numerical ODE methods like Runge--Kutta and the
superposition integral \eq{superpos} is approximated by
the trapezoidal rule. 
There are also Eulerian methods 
\cite{LeuQiaBur:07,JinWuYang:08,JinEtal:10} in which
PDEs are solved to get
the Taylor coefficients on fixed grids. 
For more discussions of numerical methods using Gaussian beams,
see \cite[Sections 8--9]{JinMarSpa:12}.

The topic of this paper is the accuracy of Gaussian beam approximations
in terms of the wavelength $\varepsilon$. 
Several such studies have been carried out
in recent years. One of the reasons have been 
to give a rigorous foundation for the beam based numerical methods above.
For the time-dependent case
error estimates were first derived for the initial data \cite{Klimes:86,Tanushev:08},
and later for the solution of scalar hyperbolic equations and the
Schr\"odinger equation
\cite{LiuRalston:09,LiuRalston:10,LiuRunborgTanushev:10,Zhen:14,LiuPryporov:15}. 
For the Helmholtz equation estimates have been 
given in \cite{MotamedRunborg:09,LiuEtAl:14}.
The general result is
that the error between the exact solution and the
Gaussian beam approximation decays as $\varepsilon^{k/2}$
for $k$-th order beams in the appropriate Sobolev
norm. However, in the recent paper \cite{Zhen:14},
Zheng showed the improved rate $\varepsilon$ for first
order beams ($k=1$) applied to the Schr\"odinger equation.
This rate agrees with the $\varepsilon^{\lceil k/2\rceil}$ rate shown in a simplified
setting for the (pointwise) Taylor expansion error away from caustics in
\cite{MotamedRunborg:09}. These improved estimates come from
exploiting error cancellations between 
adjacent beams; the higher rate is not present for single beams. 
There are also estimates for other Gaussian
beam like superpositions, in particular for so-called
frozen Gaussians \cite{RousseSwart:09,LuYang:12}
and for the acoustic wave equation with
superpositions in phase space
\cite{BougachaEtal:09}.

In this paper we first derive error estimates in general higher order
Sobolev norms for the Schr\"odinger equation and the scalar wave equation.
The result is in \theo{ErrorEstWaveSchrod} where
we obtain a convergence rate of $\varepsilon^{k/2-s}$ for $s$-order Sobolev norms.
Since the solution oscillates with period $\varepsilon$, this reduced rate
is expected. 
The proof follows closely the proof in \cite{LiuRunborgTanushev:10} for 
the case $s=0$. Second, we derive the 
main result of this paper. It is a max norm estimate given
in \theo{MaxNormError}. 
All earlier estimates for Gaussian beam approximations that we are aware of,
have been in integrated (Sobolev) norms. 
We believe this is the first max norm estimate.
We show that, away from caustics, 
the error has, uniformly,
the faster rate $\varepsilon^{\lceil k/2\rceil}$ shown in
\cite{MotamedRunborg:09,Zhen:14}, which we think is the optimal rate.
Close to caustics, our estimate
degenerates and we only get the dimensional dependent rate
$\varepsilon^{(k-n)/2}$. This rate can likely be improved, at least
for certain types of caustics, and a better understanding of this
error will be the subject of future
research. Finally, away from the essential support of the solution
the error, as well as the solution itself, decays at a spectral 
rate in $\varepsilon$.

The proof of the max norm estimate uses the 
Sobolev estimates derived in the first part of the paper, 
together with Sobolev inequalities to first get a rough estimate. 
It is subsequently refined by analyzing the difference
between beam approximations of different orders. We show in \theo{DiffRep} that
the difference can be
written as a sum of oscillatory integrals with certain properties.
The main difficulty lies in making uniform estimates
of these integrals; see \theo{Iestimate}.

The paper is organized as follows: 
In \sect{Prelim} we introduce notation and state our
main assumptions.
\sect{GBdef} introduces Gaussian
beam superpositions for the Schr\"odinger equation and the wave equation.
In \sect{GBProp} we show some simple consequences of our assumptions
as well as some known results about Gaussian beams.
\sect{Sobolev} and \sect{Maxnorm} are then devoted to proving
the error estimates in Sobolev norms and max norm, respectively.


\section{Preliminaries}\lbsec{Prelim}


In this section we introduce some notation and
describe the assumptions made 
for the PDEs and their initial data. We also summarize
some key well-posedness results.

We write $|x|$ for the Euclidean norm of a vector $x\in\Real^n$.
However, for a multi-index $\alpha=(\alpha_1,\ldots,\alpha_n)\in \Znumbers^n_+$,
 we use the standard convention that
$|\alpha|= \alpha_1+\cdots +\alpha_n$. We frequently use the simple estimate,
$$
|x^\alpha|\leq |x|^{|\alpha|},\qquad x\in\Real^n,\quad\alpha\in\Znumbers^n_+.
$$
For a function $f:\Real^n\mapsto \Real$ we let $\nabla f(x)$ denote
its gradient, and $D^2f(x)$ its Hessian matrix. Partial derivatives of 
order $\alpha$ is written as $\partial_x^\alpha f(x)$.
For a function $f:\Real^n\mapsto \Real^n$ we denote the Jacobian
matrix by $Df(x)$.

For function spaces we let
$C^\infty_b(\Real^n)$ be the functions in 
$C^\infty(\Real^n)$ whose derivatives are all bounded.
Moreover, $H^s(\Real^n)$ denotes the usual Sobolev spaces, with $H^0(\Real^n)=L^2(\Real^n)$. For these spaces we use
the standard norm, and an $\varepsilon$-scaled norm defined as
\be{Heps}
    \|f\|_{H^s(\Real^n)} :=
\sum_{|\alpha|\leq s}    
    \left\|\partial_y^\alpha f\right\|_{L^2(\Real^n)},\qquad
    \|f\|_{H_\varepsilon^s(\Real^n)} :=
\sum_{|\alpha|\leq s}    
    \varepsilon^{|\alpha|-s}\left\|\partial_y^\alpha f\right\|_{L^2(\Real^n)}.
\ee
We finally define, for continuous $f$,
\be{SupLipdef}
   ||f||_{L^\infty(K)}:=
   \sup_{z\in{K}} |f(z)|,\qquad
   |f|_{{\rm Lip}(K)}:=\sup_{z,z'\in K}  
   \frac{|f(z)-f(z')|}{|z-z'|},
\ee
and note that 
for all $T>0$, compact set $K\subset \Real^n$
and
$f(t,z)\in C^\infty([0,T]\times K)$, 
\be{suplipfin}
   \sup_{t\in[0,T]} ||f(t,\,\cdot\,)||_{L^\infty(K)},
\qquad 
   \sup_{t\in[0,T]}|f(t,\,\cdot\,)|_{{\rm Lip}(K)},
\ee
are both finite.

We then make the following precise assumptions:
\begin{itemize}
\item[\bf (A1)] Smooth
and bounded potential; strictly positive, smooth and bounded speed of propagation,
$$
  c,V\in C^\infty_b(\Real^n),\qquad \inf_{y\in\Real^n} c(y)>0.
$$
\item[\bf (A2)] Smooth and compactly supported initial amplitudes,
$$
 B_{{\ell}}\in C^\infty(\Real^n),\quad  {\rm supp}\,B_{{\ell}}\subset K_0,
 \qquad \ell=0,1,
$$
where $K_0\subset \Real^n$ is a compact set.
\item[\bf (A3)] Smooth initial phase,
$$
  \varphi_0\in C^\infty(\Real^n).
$$
For the wave equation we also assume that the 
initial phase gradient is bounded away from zero,
$$
\inf_{y\in\Real^n} |\nabla\varphi_0(y)|>0.
$$
\item[\bf (A4)] High frequency,
$$
0<\varepsilon\leq 1.
$$
\end{itemize}
These assumptions imply that there are unique, smooth, solutions of \eq{Schrodinger} and
\eq{WaveEq}.
To be precise, the solutions and their time-derivatives
belong to $L^\infty([0,T];H^s(\Real^n))$ for all $s\geq 0$ and $T>0$.
%

The corner stone of our error estimates are the energy estimates for 
 the PDEs. 
 To facilitate the presentation we will use the following notation
 for the partial differential operators,
 \be{PDOs}
    P[u] := u_{tt} - c(y)^2\Delta u,\qquad
    P^\varepsilon[u] := -i\varepsilon u_{t} 
    -\frac{\varepsilon^2}{2}\Delta u
    +V(y)u.
 \ee
The estimate of the solution of the Schr\"odinger equation uses
the norm in \eq{Heps}.
For $s\geq 0$ and $T>0$, there is a constant $C_s(T)$ such that
whenever $\varepsilon\in(0,1]$,
\be{SchWellPosed}
\sup_{0\leq t\leq T} \|u(t,\,\cdot\,)\|_{H^s} \leq 
C_s(T)\left(
    ||u(0,\,\cdot\,)||_{H^{s}_\varepsilon(\Real^n)}+
   \frac{1}{\varepsilon} \sup_{0\leq t \leq T} ||P^\varepsilon[u](t,\,\cdot\,)||_{H^{s}_\varepsilon(\Real^n)}
    \right).
\ee
This estimate is standard for $s=0$. For $s>0$
it follows by induction 
upon differentiating the Schr\"odinger equation $s$ times.
For the wave equation there is a constant $C_s(T)$ 
for each $s\geq 1$ and $T>0$, such that
\begin{align}
\lbeq{WaveWellPosed}
\lefteqn{\sup_{0\leq t\leq T}\left(
    \|u(t,\,\cdot)\|_{H^{s}(\Real^n)} +
    \|\partial_t u(t,\,\cdot)\|_{H^{s-1}(\Real^n)}\right)}
    \hspace{7 mm} &\ \\
&\leq 
    C_s(T)\left(
    \|u(0,\,\cdot\,)\|_{H^{s}(\Real^n)} +
    \|\partial_t u(0,\,\cdot\,)\|_{H^{s-1}(\Real^n)}
     + 
     \sup_{0\leq t \leq T}
     \|P[u](t,\,\cdot\,)\|_{H^{s-1}(\Real^n)}
    \right).\nonumber
\end{align}
See e.g.\mbox{} \cite[Lemma~23.2.1]{HormanderIII}.

\begin{remark}
For the Schr\"odinger equation, 
we do not need to assume the lower bound on $|\nabla\varphi_0|$.
\end{remark}
\begin{remark}
The assumption of $C^\infty$ smoothness for all functions is made for simplicity
to avoid an overly technical discussion about precise regularity requirements. In
this sense, the error estimates given below can be sharpened, since 
they will be true also for less
regular functions.
\end{remark}


\section{Gaussian Beams}\lbsec{GBdef}


In this section, we briefly describe the Gaussian beam approximation.
We restrict the description to the points that are relevant for the
accuracy analysis in subsequent sections.
For a more detailed account with a general derivation for
hyperbolic equations, dispersive wave equations and Helmholtz equation,
we
refer to 
\cite{Ralston:82,Tanushev:08,LiuRalston:09,LiuRalston:10,JinMarSpa:12,LiuRunborgTanushev:10,LiuEtAl:14}.

Individual Gaussian beams concentrate around a {\em central ray} in
space-time. 
We denote the $k$-th order Gaussian beam and the central ray
starting at $z\in K_0$ by $v_k(t,y,z)$ and
$x(t,z)$ respectively.
The beam has the following form,
\be{vdef}
v_{k}(t,y,z) =
A_{k}(t,y-x(t,z),z)e^{i\Phi_{k}(t,y-x(t,z),z)/\varepsilon},
\ee
where
\be{phidef}
  \Phi_{k}(t,y,z) = 
  \phi_{0}(t,z) + y\cdot p(t,z) + \frac12y\cdot  M(t,z) y + 
  \sum_{|\beta|=3}^{k+1} \frac1{\beta!}\phi_{\beta}(t,z) y^\beta,
\ee
and
\be{Adef1}
  A_{k}(t,y,z)= \sum_{j=0}^{\lceil \frac{k}{2} \rceil -1} \varepsilon^j \bar{a}_{j,k}(t,y,z),
\ee
\be{Adef2}
  \bar{a}_{j,k}(t,y,z)  = \sum_{|\beta|=0}^{k-2j-1} \frac1{\beta!}a_{j,\beta}(t,z) y^\beta \ .
\ee
Note that none of $\phi_{0}$, $p$, $M$, $\phi_{\beta}$ or 
$a_{j,\beta}$
depend on $k$.

Single beams are summed together to form the
$k$-th order Gaussian beam superposition solution $u_k(t,y)$,
\be{udef}
u_{k}(t,y) = 
    \left(\frac{1}{2\pi\varepsilon}\right)^\frac{n}{2} 
   \int_{K_0} v_{k}(t,y,z) \varrho_\eta(y-x(t,z)) dz,
\ee
where the integration in $z$ is over the support of the initial data $K_0\subset\Real^n$.
The function $\varrho_\eta\in C^\infty(\Real^n)$ 
is a real-valued {\em cutoff} function with radius $0<\eta\leq\infty$ satisfying, 
\begin{align}\lbeq{cutoff}
\varrho_\eta(z) \geq 0 \quad \mbox{ and } \quad \varrho_\eta(z)= 
\left\{ \begin{array}{ll}\begin{array}{l} 
1 \mbox{ for } |z|\leq\eta, \\ 
0 \mbox{ for } |z|\geq 2\eta, 
\end{array} 
&\mbox{ for } 0<\eta<\infty, \\ 
\begin{array}{l}
1,
\end{array} & \mbox{ for } \eta=\infty. \end{array}\right. 
\end{align}
As shown below in \lem{GBphaseOKforQ},
if $\eta>0$ is sufficiently small, it is ensured that 
$\Im\Phi_k>0$ on the support of $\varrho_\eta$ and
the Gaussian beam superposition is well-behaved.
For first order beams, $k=1$, the cutoff function is not needed
and we can take $\eta=\infty$.

Since the wave equation \eq{WaveEq} is a second order equation
two modes and two Gaussian beam superpositions are needed,
one for forward and one for backward propagating waves. We
denote the corresponding coefficients
by a $+$ and $-$ superscript, respectively, and
write
\be{udefwave}
u_{k}(t,y) = 
    \left(\frac{1}{2\pi\varepsilon}\right)^\frac{n}{2} 
   \int_{K_0} [v^+_{k}(t,y,z) + v^-_{k}(t,y,z)] \varrho_\eta(y-x(t,z)) dz,
\ee
where $v_k^{\pm}$ are built from the central rays $x^\pm(t,z)$ and
coefficients 
 $\phi_{0}^\pm$, $p^\pm$, $M^\pm$, $\phi^\pm_{\beta}$,
$a^\pm_{j,\beta}$.

\subsection{Governing ODEs}\lbsec{GBODE}

The central rays $x(t,z)$ and
all the coefficients $\phi_{0}$, $p$, $M$, $\phi_{\beta}$ and
$a_{j,\beta}$ satisfy ODEs in $t$. The dependence on $z$ is only via the
initial data.

For the Schr\"odinger equation
the leading order ODEs are
\begin{subequations}\lbeq{GBODESch}
\begin{align}
   \partial_t x &= p,\\
   \partial_t p &= -\nabla V(x),\\
   \partial_t \phi_{0} & =\frac{|p|^2}{2} - V(x),\\
   \partial_t M & = - M^2 - D^2V(x),\\
   \partial_t a_{0}
&=
-\frac12 {\rm Tr}(M) a_0.
\end{align}
\end{subequations}
The ODEs for the higher order coefficients
$\phi_{\beta}$ and
$a_{j,\beta}$ are more complicated.
The phase derivatives $\phi_{\beta}$
can be solved recursively in such a way that all ODEs are linear.
They are of the form
$$
\partial_t\phi_{\beta} =
-\frac12\sum_{j=1}^n \mathop{\sum_{|\gamma|=1}}_{\gamma\leq \beta}^{|\beta|-1}
 \frac{
   \beta!}{(\beta-\gamma)!\gamma!} \phi_{\beta-\gamma+e_j}\phi_{\gamma+e_j}
   -\partial_y^\beta V, \qquad |\beta| \geq 3\ .
$$
The amplitude terms $a_{j,\beta}$ satisfy a big
linear system of ODEs of the form
\be{amplitudeODE}
   \partial_t{\mathbf a} (t,z)  = {\mathbf A}(t,z){\mathbf a} (t,z),
\ee
where ${\mathbf a}$ is a vector containing all coefficients
$\{a_{j,\beta}\}$
 and ${\mathbf A}$ is a matrix determined 
from the phase terms $\{\phi_{\beta}\}$. Moreover, ${\mathbf A}$
is lower block triangular if the elements of ${\mathbf a}$
is ordered with increasing $|\beta|$;
$\partial_t a_{j,\beta}$ only depends on $a_{j,\beta'}$ with $|\beta'|\leq |\beta|$.
We refer to \cite{Ralston:82,Tanushev:08} for more detailed
discussions.

The leading order ODEs for the two modes of the wave equation are 
\begin{subequations}\lbeq{GBODEWave}
\begin{align}
                  \partial_t x^\pm &= \pm c(x^\pm)\frac{p^\pm}{|p^\pm|},\\
                  \partial_t p^\pm &= \mp \nabla c(x^\pm) |p^\pm|,\\
    \partial_t \phi^\pm_0 & = 0,\\
\partial_t M^\pm & = \mp(E + B^T M^\pm + M^\pm B + M^\pm C M^\pm), \\
\partial_t a^\pm_0 & = \pm \frac{1}{2 |p^\pm|} \left(
p^\pm\cdot \nabla c(x^\pm)
+\frac{c(x^\pm)\: p^\pm \cdot M  p^\pm}{|p^\pm|^2}
-c(x^\pm) \: \text{Tr}(M^\pm) 
\right) a^\pm_0,
    \end{align}
\end{subequations}
with
\begin{equation*}
        E= |p^\pm| D^2 c(x^\pm), \qquad B = \frac{p^\pm \otimes \nabla c(x^\pm)}{|p^\pm|}, \qquad
        C = \frac{c(x^\pm)}{|p^\pm|}{\rm Id}_{n\times n} - \frac{c(x^\pm)}{|p^\pm|^3}  p^\pm \otimes p^\pm.
       \end{equation*}
The higher order phase terms $\{\phi^\pm_\beta\}$  again satisfy linear ODEs, if
solved in the right order, and the higher order amplitude
terms 
$\{a^\pm_{j,\beta}\}$
satisfy a linear ODE system of the type \eq{amplitudeODE}.

\begin{remark}\lbrem{hamiltonian}
The leading order ODEs for both equations, and for general hyperbolic
equations, actually have a Hamiltonian structure,
\begin{subequations}\lbeq{GBODEGen}
\begin{align}
   \partial_t x &= \nabla_p H(x,p),\\
   \partial_t p &= -\nabla_x H(x,p),\\
   \partial_t \phi_{0} & = -H(x,p) + p\cdot\nabla_pH(x,p),
\end{align}
\end{subequations}
where $H=|p|^2/2+V(x)$ for the Schr\"odinger equation and
$H= \pm c(x)|p|$ for the two modes of the wave equation.
\end{remark}

\subsection{Initial Data}\lbsec{GBini}

Each Gaussian beam
$v_{k}(t,y,z)$ requires initial values
for the central ray and all of the amplitude and phase Taylor coefficients. The
appropriate choice of these initial values will make $u_k(0,y)$
asymptotically converge to the initial conditions in 
\eq{Schrodinger} and
\eq{WaveEq}.
As shown in \cite{LiuRunborgTanushev:10}, initial data for
the central ray and phase coefficients should be chosen as
follows, for the Schr\"odinger as well as the two modes of the
wave equation.
\begin{subequations}\lbeq{GBini}
\begin{align}
 x(0,z) &= z, \\ 
 p(0,z) &= \nabla \varphi_0(z), \\
 \phi_0(0,z) &= \varphi_0(z), \\
 M(0,z) &= D^2\varphi_0(z) + i\ {\rm Id}_{n\times n},\\ 
 \phi_{\beta}(0,z) &= \partial_y^\beta\varphi_0(z), \qquad |\beta|= 3,\ldots,k+1 \ .
\end{align}
\end{subequations}
For the Schr\"odinger equation,
initial values for the amplitude coefficients should be given as
\be{GBiniA}
 a_{j,\beta}(0,z) = \begin{cases} \partial_y^\beta B_0(z), & j=0,\\ 
 0, & j>0.
 \end{cases}
\ee
The construction is more complicated for the wave equation. Let
\begin{align*}
   \bar{A}^\pm_0(y,z) &= \frac12\left( B_0(y) + \frac{B_1(y)}{id_t\Phi^{\pm}_k(0,y-z,z)}\right),\\
   \bar{A}^\pm_{j+1}(y,z) &= -\frac12\frac{d_t\bar{a}^+_{j,k}(0,y-z,z)+
   d_t\bar{a}^-_{j,k}(0,y-z,z)}{id_t\Phi^{\pm}_k(0,y-z,z)},
\end{align*}
where
\begin{align*}
d_t\Phi^{\pm}_k(0,y-z,z)&:=
\partial_t\Phi^{\pm}_k(0,y-z,z)-
\partial_t x^{\pm}(0,z)\cdot \nabla_y\Phi^{\pm}_k(0,y-z,z),\\
d_t\bar{a}^{\pm}_{j,k}(0,y-z,z)&:=
\partial_t\bar{a}^{\pm}_{j,k}(0,y-z,z)-
\partial_t x^{\pm}(0,z)\cdot \nabla_y\bar{a}^{\pm}_{j,k}(0,y-z,z).
\end{align*}
Then
\be{GBiniAwave}
 a^\pm_{j,\beta}(0,z) = \partial_y^\beta\bar{A}^\pm_{j}(y,z)|_{y=z}.
\ee
Note that the time derivatives $\partial_t\Phi_k^\pm$, $\partial_t x^\pm$
and $\partial_t\bar{a}^{\pm}_{j,k}$ are given by the right
hand side of the ODE system.


\section{Gaussian Beam Properties}\lbsec{GBProp}


In this section we collect some simple consequences
of assumptions (A1)--(A4) 
 for the 
Gaussian beam approximations, as well as some
other known results.

\subsection{Existence and Regularity}
\lbsec{GBexist}

From (A1) and (A3) it follows that
the Gaussian beam coefficient functions are well-defined
for all times $t\geq 0$ and initial positions $z\in\Real^n$.
We briefly motivate why.
By (A1) the right hand sides of the
ODEs for $(x(t,z),\, p(t,z))$ are globally Lipschitz, for
the Schr\"odinger equation.
For the two modes of the wave equation we use (A3) and
the fact that the Hamiltonian $\pm c(x)|p|$
is constant along the flow. From this it follows that for all $t$,
$$
0<p_{\rm min}:=\frac{c_{\rm min}}{c_{\rm max}}
\inf_{y\in\Real} |\nabla\varphi_0(y)|
\leq
|p^\pm(t,z)|\leq
\frac{c_{\rm max}}{c_{\rm min}}
|\nabla\varphi_0(z)|=:p_{\rm max}(z)<\infty,
$$ 
where $c_{\min}=\inf c(y)$ and $c_{\max}=\sup c(y) $.
The right hand sides of the ODE
for $(x^\pm(t,z),\, p^\pm(t,z))$ are globally Lipschitz for these
values of $p^\pm$.
It follows that unique solutions to the ODEs exist for all times.
Moreover, the choice of initial data and 
a result in \cite[Section 2.1]{Ralston:82} ensure that
the non-linear Riccati equations
for $M$ and $M^\pm$ also have solutions for all times.
The remaining coefficient functions are well-defined since they
satisfy linear ODEs with variable, continuous, coefficients.

Furthermore, the coefficient functions are smooth functions
of $t$ and $z$.
By (A2) and (A3)
all coefficient functions are solutions to ODEs with 
initial data that is
$C^\infty(\Real^n)$ in $z$. 
The right hand sides of the ODEs are also smooth,
for both equations, since $|p^\pm|\geq p_{\rm min}>0$ for the
wave equation.
The regularity of the initial data therefore persists  for $t>0$.
Hence, 
\be{coeffsmooth}
 x,\,x^\pm,\,p,\,p^\pm,
 \,\phi_{0},\,\phi_{0}^\pm,\,M,\,M^\pm,\,\phi_{j,\beta},\,\phi_{j,\beta}^\pm,\,a_{j,\beta},\,a_{j,\beta}^\pm \in C^\infty([0,\infty)\times \Real^n),
\ee
for all $j,\beta$.
Moreover, by the form of the ODEs for the amplitude coefficients
\eq{amplitudeODE} and the fact that initial data is compactly supported,
all amplitude coefficients will be compactly supported in $z$ for $t\geq 0$,
\be{acompact}
   {\rm supp}\,a_{j,\beta}(t,\,\cdot\,)\subset K_0,\quad
   {\rm supp}\,a^\pm_{j,\beta}(t,\,\cdot\,)\subset K_0,\qquad t\in [0,\infty).
\ee

We finally note that
none of the coefficient functions
$x$, $p$, $\phi_{0}$, $M$, $\phi_{j,\beta}$,
$a_{j,\beta}$,
and the corresponding functions for the wave equation,
depend on the order $k$ of the beam.
This is true since the ODEs and the initial data
for higher order coefficients functions
only involve lower order coefficient functions.
Hence, the higher order beams
have the same lower order coefficient functions as
the lower order beams.

\subsection{Initial Data}

For the initial data chosen as in \sect{GBini}, the following 
error estimate follows from a result in \cite{LiuRunborgTanushev:10}.
\begin{theorem}\lbtheo{InitialData}
Let $u_k$ be the Gaussian beam superposition approximation \eq{udef} to 
the Schr\"odinger equation
\eq{Schrodinger}
or \eq{udefwave} to the wave equation \eq{WaveEq}, 
with initial data determined as in \sect{GBini}.
Then, if $u$ is the corresponding exact solution,
there is a constant $C$ such that
\be{uinitest}
  \left\|  u_k(0,\,\cdot\,) - u(0,\,\cdot\,)  \right\|_{H^s} \leq
  \left\|  u_k(0,\,\cdot\,) - u(0,\,\cdot\,)  \right\|_{H_{\varepsilon}^s} \leq C \varepsilon^{\frac{k}{2}-s} \ ,\qquad \forall\varepsilon\in(0,1],
\ee
and, for the wave equation, 
\be{utinitest}
 \left\| \partial_tu_k(0,\,\cdot\,) - \partial_tu(0,\,\cdot\,)] \right\|_{H^{s-1}} \leq C \varepsilon^{\frac{k}{2}-s} \ ,
 \forall\varepsilon\in(0,1],
\ee
for $s\geq 1$.
\end{theorem}
\begin{proof}
It was shown in 
\cite[Lemma 3.6]{LiuRunborgTanushev:10} that
there are constants $C_{0,\alpha}$
and $C_{1,\alpha}$ such that
\begin{align*}
 \left\| \partial_y^\alpha u_k(0,\,\cdot\,) - \partial_y^\alpha u(0,\,\cdot\,) \right\|_{L^2} &\leq C_{0,\alpha} \varepsilon^{\frac{k}{2}-|\alpha|} \ ,
\end{align*}
and, for the wave equation \eq{WaveEq}
\begin{align*}
 \left\| \partial_y^\alpha \partial_t u_k(0,\,\cdot\,) -  \partial_y^\alpha\partial_tu(0,\,\cdot\,) \right\|_{L^2} &\leq C_{1,\alpha} \varepsilon^{\frac{k}{2}-|\alpha|-1} \ .
\end{align*}
Clearly $||\cdot||_{H^s}\leq ||\cdot||_{H_\varepsilon^s}$ when $\varepsilon\leq 1$,
and
from the definition in \eq{Heps},
\begin{align*}
\left\| u_k(0,\,\cdot\,) - u(0,\,\cdot\,)  \right\|_{H_{\varepsilon}^s}
&=
\sum_{|\alpha|\leq s}    
    \varepsilon^{|\alpha|-s}\left\|\partial_y^\alpha
    u_k(0,\,\cdot\,) - \partial_y^\alpha u(0,\,\cdot\,) \right\|_{L^2(\Real^n)}\\
    &\leq
    \varepsilon^{\frac{k}{2}-s}
\sum_{|\alpha|\leq s} C_{0,\alpha}   
=:    C \varepsilon^{\frac{k}{2}-s}.
\end{align*}
This shows \eq{uinitest}.
The estimate \eq{utinitest} follows in a similar way.
\end{proof}

\subsection{Phase and Ray Properties}

The Gaussian beam phases and central rays
have the following properties, as shown in \cite[Lemma 3.4]{LiuRunborgTanushev:10}.
\begin{lemma}\lblem{GBphaseOKforQ}
Under assumptions (A1)--(A4), 
for a given compact set $K_0\subset\Real^n$, final time $T>0$
and beam order $k$, there is a Gaussian beam cutoff width $\eta_0>0$ such that
the Gaussian beam phase $\Phi$ and central ray $x$ have
the following properties for all $0<\eta\leq \eta_0$:
\begin{itemize}
\item[{\bf (P1)}] $x(t,z)\in C^\infty([0,T]\times\Real^n)$,
\item[{\bf (P2)}] $\Phi(t,y,z)\in C^\infty([0,T]\times\Real^n\times\Real^n)$, 
\item[{\bf (P3)}]
$\nabla\Phi(t,0,z)$ is real and 
there is a constant $C$ such that 
\begin{align*}
  |\nabla_y\Phi(t,0,z)-\nabla_y\Phi(t,0,z')|+|x(t,z)-x(t,z')|\geq C|z-z'|\ ,
\end{align*}
for all  $t\in[0,T]$ and $z,z'\in K_0$.

\item[{\bf (P4)}]
there exists a constant $w_4>0$ such that
\begin{align*}
  \Im\Phi(t,y,z)\geq w_4|y|^2\ ,\qquad \forall t\in[0,T],\ z\in K_0,
\end{align*}
when $|y|\leq 2\eta$ (or for all $y$ if $\eta=\infty$).
\end{itemize}
Here, $\Phi$ and $x$ can be 
either 
the phase and central ray
of  the Schr\"odinger 
equation,  $\Phi_k$ and $x$, or 
of one of the wave equation modes, 
$\Phi_k^\pm$ and $x^\pm$.
When $k=1$, 
$\eta$ can take any value in $(0,\infty]$,
that is $\eta_0=\infty$.
\end{lemma}

These properties of the phase and the central ray are of
great importance in the subsequent estimates. 
In fact, they are necessary for the Gaussian beam approximation to
be accurate.
Following this lemma we therefore make the definition:
\begin{definition}
The cutoff width $\eta$ 
used for the Gaussian beam approximation of
\eq{Schrodinger} and \eq{WaveEq}
is called {\em admissible} for $K_0$, $T$ and $\Phi$
if it is small enough in the sense of \lem{GBphaseOKforQ}.
\end{definition}

We note that if $\eta$ is admissible then
$\eta'$ is also admissible if $\eta'\leq\eta$.
Moreover, the difference between two solutions with
different admissible cutoff widths, is exponentially small in $\varepsilon$,
as seen in the following lemma.
\begin{lemma}\lblem{cutoffdiff}
If $\eta$, $\eta'$ are both admissible cutoff widths, and
$u_k$, $u_k'$ are the corresponding 
Gaussian beam superpositions for the Schr\"odinger
equation or the wave equation, then
$$
  \sup_{t\in[0,T]}||u_k(t,\,\cdot\,)-u_k'(t,\,\cdot\,)||_{L^\infty(\Real^n)}\leq C e^{-w/\varepsilon},
$$
for some constants $C$ and $w>0$.
\end{lemma}
\begin{proof}
We consider the Schr\"odinger case.
Suppose $\eta'<\eta\leq\infty$.
From
 the construction of beams in \sect{GBdef} 
 together with 
 \eq{SupLipdef} and \eq{coeffsmooth},
there is a constant $C$ such
that $|A_k(t,y,z)|\leq C(1+|y|^{k-1})$ for all
$t\in[0,T]$, $z\in K_0$ and $\varepsilon\in(0,1]$.
Then using (P4) in \lem{GBphaseOKforQ}, with $t\in[0,T]$,
\begin{align*}
|u_k(t,y)-u_k'(t,y)| &= 
    \left(\frac{1}{2\pi\varepsilon}\right)^\frac{n}{2} 
   \left|\int_{K_0} v_{k}(t,y,z) [\varrho_\eta(y-x(t,z))-\varrho_{\eta'}(y-x(t,z))] dz
\right|\\
&\leq 
    \left(\frac{1}{2\pi\varepsilon}\right)^\frac{n}{2} 
   \int_{K_0\setminus \{z\,;\,|y-x|\leq \eta'\}}  |v_{k}(t,y,z)|dz\\
&=
    \left(\frac{1}{2\pi\varepsilon}\right)^\frac{n}{2} 
   \int_{K_0\setminus \{z\,;\,|y-x|\leq \eta'\}} |A_k(t,y-x,z)| e^{-\Im\Phi(t,y-x,z)/\varepsilon} dz\\
&\leq C'\varepsilon^{-n/2}
   \int_{K_0\setminus \{z\,;\,|y-x|\leq \eta'\}} \left(1+ |y-x|^{k-1}\right)
   e^{-w_4|y-x|^2/\varepsilon} dz.
\end{align*}
We now use the fact that for given $p\geq 0$ and $c>0$ there is a constant $D$
such that $|x|^p\exp(-cx^2/\varepsilon)\leq D\exp(-cx^2/2\varepsilon)$ for all $x$.
Then,
\begin{align*}
|u_k(t,y)-u_k'(t,y)| &\leq
C'\varepsilon^{-n/2}
\int_{K_0\setminus \{z\,;\,|y-x|\leq \eta'\}}(1+D)
   e^{-w_4|y-x|^2/2\varepsilon} dz\\
&\leq C'\varepsilon^{-n/2}|K_0|(1+D)e^{-w_4{\eta'}^2/2\varepsilon}
\leq C''e^{-w/\varepsilon},
\end{align*}
for some $0<w< w_4{\eta'}^2/2$.
The wave equation case is proved by considering each mode
separately, in the same way.
\end{proof}

\subsection{Representation with Oscillatory Integrals}

An important step in the Gaussian beam error estimates
in \cite{LiuRunborgTanushev:10} is to bound the residual 
that appears when
the Gaussian beam approximation is entered into the PDE.
Up to a small term in $\varepsilon$, this residual can be written as
a sum of oscillatory integrals 
belonging to a family defined as follows.
For a phase $\Phi$, central ray $x$, multi-index $\alpha$, 
compact set $K_0\subset\Real^n$,
cutoff function $\varrho_\eta$ as given in
\eq{cutoff} 
and a continuous function $g(t,y,z,\varepsilon)$, we let
\begin{align}\lbeq{IntDef}
   & \Int^\alpha_{\Phi,x,g}(t,y) \\ & \qquad := \varepsilon^{-\frac{n+|\alpha|}{2}}\int_{K_0} g(t,y,z,\varepsilon)(y-x(t,z))^\alpha e^{i\Phi(t,y-x(t,z),z)/\varepsilon} \varrho_\eta(y-x(t,z))dz \ . \notag
\end{align}
Indeed, the following lemma was shown in \cite{LiuRunborgTanushev:10}.
\begin{lemma}
\lblem{PuInTermsOfQ} 
  Under assumptions (A1)--(A4) the
  Schr\"odinger operator $P^\varepsilon$ and the
  wave equation operator $P$ 
  in \eq{PDOs}
  acting on the Gaussian beam superposition $u_k$
  can be accurately approximated by a finite sum of oscillatory integrals
  of the type \eq{IntDef},
\begin{align*}
P^\varepsilon[u_k](t,y) &=
\varepsilon^{\frac{k}{2}+1} \sum_{j=1}^{J} \varepsilon^{\ell_j}\Int^{\alpha_j}_{\Phi_k,x,g_j}(t,y)  +\Ordo(\varepsilon^\infty), \\
P[u_k](t,y) &=
\varepsilon^{\frac{k}{2}-1} \sum_{j=1}^{J} \varepsilon^{\ell_j}
\left({\Int}
^{\alpha_j}_{\Phi^+_k,x^+,g^+_j}(t,y)+
{\Int}
^{\alpha_j}_{\Phi^-_k,x^-,g^-_j}(t,y)\right)  +\Ordo(\varepsilon^\infty),
\end{align*}
  where $\ell_j\geq 0$,
  and $\eta$ is assumed to be admissible for $K_0$, $T$ and the 
  corresponding Gaussian beam phase(s), $\Phi_k$ 
  or $\Phi_k^\pm$. Moreover,
 $(\Phi_k,x)$ or $(\Phi_k^\pm,x^\pm)$,
  have properties (P1)--(P4), and all
  $g_j$, $g_j^\pm$ have the following property:
\begin{itemize}
\item[{\bf (P5)}]
$g(t,y,z)\in C^\infty([0,T]\times \Real^n\times K_0)$ is independent
of $\varepsilon$ and
 for any multi-index $\beta$ 
 there exists a constant $C_\beta$
such that
\begin{align*}
   \sup_{y\in\Real^n} \left|\partial_y^\beta g(t,y,z)\right|\leq C_\beta\
   \ ,\qquad \forall t\in[0,T],\ z\in K_0.
\end{align*}
\end{itemize}
\end{lemma}
\begin{remark}\lbrem{einfterm}
A closer inspection of the proof of this lemma in \cite{LiuRunborgTanushev:10}
reveals that also the derivatives with respect to $(t,y)$ 
of the exponentially small terms
$\Ordo(\varepsilon^\infty)$ are exponentially small in $\varepsilon$.
\end{remark}

The key estimate in \cite{LiuRunborgTanushev:10} used to
bound the residuals $P^\varepsilon[u_k]$ and $P^\varepsilon[u]$
is the following theorem, which gives an $\varepsilon$-independent
$L^2$ estimate of the integrals in \eq{IntDef}.
\begin{theorem}\lbtheo{IestL2}
If the phase $\Phi$ and central ray $x$ have properties (P1)--(P4),
and $g$ has property (P5), then there is a constant $C$
such that, for all $\varepsilon\in(0,1]$,
\be{Il2}
   \sup_{t\in[0,T]}\left\|\Int^\alpha_{\Phi,x,g}(t,\,\cdot\,)\right\|_{L^2} \leq C.
\ee
\end{theorem}
In Theorem 3.2 in \cite{LiuRunborgTanushev:10},
an integral operator of the same
form was estimated. That result immediately gives \eq{Il2}.


\section{Error Estimates in Sobolev Norms}\lbsec{Sobolev}


Here we show the following theorem.
\begin{theorem}\lbtheo{ErrorEstWaveSchrod}
Let $u_k$ be
the $k$-th order Gaussian beam
superposition given in \sect{GBdef} for the Schr\"odinger \eq{Schrodinger}
equation or the wave equation \eq{WaveEq}, with an
$\eta$ that is admissible for $K_0$, $T>0$ and the 
corresponding Gaussian beam phases,
$\Phi_k$ or $\Phi^\pm_k$.
If $u$ is the exact solution to Schr\"odinger's equation 
\eq{Schrodinger} and $s\geq 0$, there is a constant $C$ such that
\be{SchrodingerError}
 \sup_{0\leq t \leq T}||u(t,\,\cdot\,)-u_k(t,\,\cdot\,)||_{H^s(\Real^n)} \leq C\varepsilon^{\frac{k}{2}-s}\ ,\qquad \forall \varepsilon\in(0,1].
\ee
If $u$ is the exact solution to the wave equation 
\eq{WaveEq} and $s\geq 1$, there is a constant $C$ such that
\be{WaveError}
\sup_{0\leq t\leq T}\left(
    \|u_k(t,\,\cdot)-u(t,\,\cdot)\|_{H^{s}(\Real^n)} +
    \|\partial_t u_k(t,\,\cdot)-\partial_t u(t,\,\cdot)\|_{H^{s-1}(\Real^n)}\right)
    \leq 
C\varepsilon^{\frac{k}{2}-s}\ ,
\ee
for all $\varepsilon\in(0,1]$.
\end{theorem}

The results \eq{SchrodingerError} with $s=0$ and
\eq{WaveError} with $s=1$
were proved earlier in
\cite{LiuRunborgTanushev:10}. This theorem extends the results
to higher order Sobolev norms.
Note that $\varepsilon^{-s}$ is the rate at which the norm of the 
initial data for the PDEs
go to infinity as $\varepsilon\to 0$, because of their oscillatory nature.
The decreased rate for larger $s$ is therefore expected also for the
solution error. Still, for large enough $k$ the Gaussian beam
approximation will converge as $\varepsilon\to 0$ also in higher 
order Sobolev norms.

We now prove the results for the two types of PDEs separately. 
For the Schr\"odinger equation \eq{Schrodinger}, applying the
well-posedness estimate given in 
\eq{SchWellPosed}
to the difference
between the true solution $u$ and the $k$-th order Gaussian beam superposition,
$u_k$ we obtain
\begin{align*}
\lefteqn{\sup_{0\leq t\leq T} \|u_k(t,\,\cdot\,)-u(t,\,\cdot\,)\|_{H^s(\Real^n)}}
\hspace{15 mm} &\\ &\leq 
C_s(T)\left(
    ||u_k(0,\,\cdot\,)-u(0,\,\cdot\,)||_{H^{s}_\varepsilon(\Real^n)}+
   \frac{1}{\varepsilon} \sup_{0\leq t \leq T} ||P^\varepsilon[u_k](t,\,\cdot\,)||_{H^{s}_\varepsilon(\Real^n)}
    \right).
\end{align*}
The first term of the right hand side, which represents the difference in the initial data, can be estimated by \theo{InitialData} and the second term, which represents the evolution error, can be rewritten using \lem{PuInTermsOfQ} 
and then estimated to obtain
\begin{align}\lbeq{sest1}
\lefteqn{\sup_{0\leq t\leq T} \|u_k(t,\,\cdot\,)-u(t,\,\cdot\,)\|_{H^s}}
\hspace{15 mm} &\\ &\leq 
C_s(T)\left(
C\varepsilon^{\frac{k}{2}-s}+
\sup_{0\leq t \leq T} 
   \varepsilon^{\frac{k}{2}} \sum_{j=1}^{J}
  \left\|
   \Int^{\alpha_j}_{\Phi_k,x,g_j}(t,\,\cdot\,)\right\|_{H^{s}_\varepsilon(\Real^n)}
    \right)  +\Ordo(\varepsilon^\infty),\nonumber
\end{align}
since $\ell_j\geq 0$ in \lem{PuInTermsOfQ}.
Here we also used \rem{einfterm}, which implies that the 
Sobolev norm of $O(\varepsilon^\infty)$ is again $O(\varepsilon^\infty)$.

 To continue,
 we need to estimate $\Int^{\alpha_j}_{\Phi,x,g_j}$ in Sobolev norms.
 In \theo{IestL2}, such estimates were given in $L^2$-norm. In \sect{derivQL2norm}, we extend this result
to general
Sobolev spaces by proving the following theorem:
\begin{theorem}\lbtheo{derivQL2norm}
If the phase $\Phi$ and central ray $x$ have properties (P1)--(P4),
and $g$ has property (P5), then there is a constant $C$ 
such that, for all $\varepsilon\in(0,1]$,
$$
   \sup_{t\in[0,T]}\left\|\Int^\alpha_{\Phi,x,g}(t,\,\cdot\,)\right\|_{H^s(\Real^n)}
   \leq
   \sup_{t\in[0,T]}\left\|\Int^\alpha_{\Phi,x,g}(t,\,\cdot\,)\right\|_{H^s_\varepsilon(\Real^n)}
   \leq C \varepsilon^{-s}.
$$
\end{theorem}
Upon applying \theo{derivQL2norm} to \eq{sest1} we
obtain \eq{SchrodingerError}.

For the wave equation \eq{WaveEq}
we use \eq{WaveWellPosed} and obtain
\begin{align}\lbeq{iest0}
\lefteqn{\sup_{0\leq t\leq T}\left(
    \|u_k(t,\,\cdot)-u(t,\,\cdot)\|_{H^{s}(\Real^n)} +
    \|\partial_t u_k(t,\,\cdot)-\partial_t u(t,\,\cdot)\|_{H^{s-1}(\Real^n)}\right)}
    \hspace{7 mm} &\ \nonumber\\
&\leq 
    C_s(T)\Bigl(
    \|u_k(0,\,\cdot\,)-u(0,\,\cdot\,)\|_{H^{s}(\Real^n)} +
    \|\partial_t u_k(0,\,\cdot\,)-\partial_t u(0,\,\cdot\,)\|_{H^{s-1}(\Real^n)}
    \nonumber\\
    &\ \ \ \
     + 
     \sup_{0\leq t \leq T}
     \|P[u_k](t,\,\cdot\,)\|_{H^{s-1}(\Real^n)}
    \Bigr).
\end{align}
From \theo{InitialData} 
we can again estimate the initial data terms,
\be{iest1}
    \|u_k(0,\,\cdot\,)-u(0,\,\cdot\,)\|_{H^{s}(\Real^n)} +
    \|\partial_t u_k(0,\,\cdot\,)-\partial_t u(0,\,\cdot\,)\|_{H^{s-1}(\Real^n)}
\leq C\varepsilon^{\frac{k}{2}-s}.
\ee
Moreover, by \lem{PuInTermsOfQ}, \rem{einfterm} and \theo{derivQL2norm}
\begin{align}\lbeq{iest2}
\lefteqn{\sup_{0\leq t \leq T}
     \|P[u_k](t,\,\cdot\,)\|_{H^{s-1}(\Real^n)}}\hskip 5 mm &\nonumber\\
     &\leq 
     \varepsilon^{\frac{k}{2}-1}\sum_{j=1}^{J} \varepsilon^{\ell_j}
          \Bigl(
          \sup_{0\leq t \leq T}
          \left\|{\Int}
^{\alpha_j}_{\Phi^+_k,x^+,g^+_j}(t,\,\cdot\,)\right\|_{H^{s-1}(\Real^n)}\nonumber\\
&\ \ \ \ \ \ \ \ +
\sup_{0\leq t \leq T}
\left\|{\Int}
^{\alpha_j}_{\Phi^-_k,x^-,g^-_j}(t,\,\cdot\,)\right\|_{H^{s-1}(\Real^n)}\Bigr)
       +\Ordo(\varepsilon^\infty)
       \nonumber\\
       &\leq 
     \varepsilon^{\frac{k}{2}-1}\sum_{j=1}^{J} C\varepsilon^{\ell_j-s+1}
     \leq C\varepsilon^{\frac{k}{2}-s}.
\end{align}
Together \eq{iest0}, \eq{iest1} and \eq{iest2} gives \eq{WaveError} and
the proof of \theo{ErrorEstWaveSchrod} is complete.
We now turn to proving \theo{derivQL2norm}. 

\subsection{Proof of \theo{derivQL2norm}}\lbsec{derivQL2norm}

The main idea of the proof is to reduce the derivative
of the oscillatory integral to a sum of the same
type of integrals, scaled by $\varepsilon$, and then apply
\theo{IestL2}.
We begin by proving a lemma giving the form of the derivatives of
a monomial multiplying the exponential of a polynomial.
\begin{lemma}\lblem{yePderiv}
Suppose $Q(y,r)$ is a polynomial in $y$ with coefficients
that depend smoothly on $r$. Then for multi-indices $\alpha$ and $\beta$,
\be{yePderiv}
   \partial_y^\beta \left(y^\alpha e^{iQ(y,r)/\varepsilon} \right)
   = \varepsilon^{|\alpha|-|\beta|}\sum_{|\gamma|=0}^{|\alpha|} \left(\frac{y}{\varepsilon}\right)^{\gamma}
   Q_{\gamma,\beta}(y,r) e^{iQ(y,r)/\varepsilon},
\ee
for some $Q_{\gamma,\beta}(y,r)$ which are also
polynomials in $y$ with coefficients
depending smoothly on $r$.
\end{lemma}
\begin{proof}
We use induction and first note that \eq{yePderiv} holds for $\beta=0$
with $Q_{\alpha,0}\equiv 1$ and $Q_{\gamma,0}\equiv 0$ for $\gamma\neq\alpha$.
Let $e_j$ be the unit vector multi-index and suppose ${\gamma}=({\gamma}_1,\ldots,{\gamma}_n)$.
Then, assuming \eq{yePderiv} holds for $\beta$,
\begin{align*}
\partial_y^{\beta+e_j} y^\alpha e^{iQ(y,r)/\varepsilon} 
   &=  \varepsilon^{|\alpha|-|\beta|} \partial_{y_j}
   \sum_{|\gamma|=0}^{|\alpha|} \left(\frac{y}{\varepsilon}\right)^{\gamma}
      Q_{\gamma,\beta}(y,r) e^{iQ(y,r)/\varepsilon}
\\
   &=
   \varepsilon^{|\alpha|-|\beta|-1}
   \sum_{|\gamma|=0}^{|\alpha|} 
   \left(\frac{y}{\varepsilon}\right)^{\gamma-e_j}
   [\gamma_jQ_{\gamma,\beta}(y,r) +
      y_j\partial_{y_j}Q_{\gamma,\beta}(y,r)]    e^{iQ(y,r)/\varepsilon}   
\\      
  &\mbox{\hskip 5 mm} +  i   \varepsilon^{|\alpha|-|\beta|-1}
  \sum_{|\gamma|=0}^{|\alpha|} \left(
  \frac{y}{\varepsilon}\right)^{\gamma}Q_{\gamma,\beta}(y,r)[\partial_{y_j}Q(y,r)]
e^{iQ(y,r)/\varepsilon}.
\end{align*}
This is of the same form as \eq{yePderiv}
if we identify $Q_{\gamma,\beta+e_j}=
iQ_{\gamma,\beta}\partial_{y_j}Q + (\gamma_j+1)Q_{\gamma+e_j,\beta}
+y_j\partial_{y_j+e_j}Q_{\gamma+e_j,\beta}$
for $|\gamma|<|\alpha|$ and
$Q_{\gamma,\beta+e_j}=
iQ_{\gamma,\beta}\partial_{y_j}Q$ when $|\gamma|=|\alpha|$.
Moreover $Q_{\gamma,\beta+e_j}(y,r)$ depends smoothly
on $r$ since $Q_{\gamma,\beta}$ and $Q$ do.
The lemma is therefore proved by induction.
\end{proof}

We now continue with the proof of \theo{derivQL2norm}.
Let 
$$
   W(t,y,z) = y^\alpha e^{i\Phi(t,y,z)/\varepsilon}.
$$
Then, since
$\Phi(t,y,z)$ is a 
$k+1$ degree polynomial in $y$ with coefficients depending smoothly
on $t$ and $z$ we can use \lem{yePderiv} to obtain
\begin{align*}
  \partial_y^\beta\Int^\alpha_{\Phi,x,g}(t,y) 
  &=
  \varepsilon^{-\frac{n+|\alpha|}{2}}\int_{K_0}
  \partial_y^\beta \Bigl(g(t,y,z)W(t,y-x(t,z),z)\varrho_\eta(y-x(t,z))\Bigr)dz \\
&=  \varepsilon^{-\frac{n+|\alpha|}{2}}
  \sum_{\beta_1+\beta_2+\beta_3=\beta}C_{\beta_1,\beta_2,\beta_3}
  \int_{K_0}
  [\partial_y^{\beta_1} g]
  [\partial_y^{\beta_2} W][\partial_y^{\beta_3} \varrho_\eta]dz\\
&=  
  \sum_{\beta_1+\beta_2+\beta_3=\beta}
  \sum_{|\gamma|=0}^{|\alpha|}C_{\beta_1,\beta_2,\beta_3}
  \varepsilon^{-\frac{n+|\alpha|}{2}}
  \varepsilon^{|\alpha|-|\beta_2|-|\gamma|} 
I_{\beta_1,\beta_2,\beta_3,\gamma}(t,z),
\end{align*}
where
\begin{align*}
 I_{\beta_1,\beta_2,\beta_3,\gamma}(t,y) 
&=  \int_{K_0}
    [\partial_y^{\beta_1} g(t,y,z)]
   (y-x(t,z))^\gamma
   Q_{\gamma,\beta_2}(t,y-x(t,z),z)
   \\
   &\mbox{\hskip 5 mm}\times
  e^{i\Phi(t,(y-x(t,z),z)/\varepsilon}  
  [\partial_y^{\beta_3} \varrho_\eta(y-x(t,z))]dz,
\end{align*}
with $Q_{\gamma,\beta_2}(t,y,z)$ being polynomials
in $y$ depending smoothly on $t$ and $z$.
We now first consider the terms $I_{\beta_1,\beta_2,\beta_3,\gamma}$
where $|\beta_3|>0$.
Since the derivatives of $\varrho_\eta(y-x(t,z))\equiv 0$ except when
$\eta\leq |y-x(t,z)|\leq 2\eta$, and by properties (P4), (P5),
$$
|I_{\beta_1,\beta_2,\beta_3,\gamma}|
\leq 
C(T)\int_{K_0}
  e^{-w_4 \eta ^2/\varepsilon}  
  dz\leq C(T) e^{-w_4\eta ^2/\varepsilon},
$$
for all $0\leq t\leq T$. The remaining terms $I_{\beta_1,\beta_2,0,\gamma}$ are all of the form
$$
  \int_{K_0} \tilde{g}(t,y,z)
  (y-x(t,z))^\gamma
  \tilde{Q}(t,y-x(t,z),z)
   e^{i\Phi(t,y-x(t,z),z)/\varepsilon} \varrho_\eta(y-x(t,z))dz, 
$$
for some smooth function $\tilde{g}$, which is a $y$-derivative of $g$, and
$\tilde{Q}(t,y,z)$ which is a polynomial in $y$ with 
coefficients that are smooth in $t$ and $z$. Suppose the
degree of $\tilde{Q}$ is $d$ and denote the coefficients by $q_\ell(t,z)$.
Then the term can be written as
\begin{align*}
I(t,y) &=\sum_{|\ell|=0}^d
 \int_{K_0} \tilde{g}(t,y,z)
  q_\ell(t,z)
(y-x(t,z))^{\gamma+\ell}
   e^{i\Phi(t,y-x(t,z),z)/\varepsilon} \varrho_\eta(y-x(t,z))dz 
\\
 &=\sum_{|\ell|=0}^d
   \varepsilon^{\frac{n+|\gamma|+|\ell|}{2}}
  \Int^{\gamma+\ell}_{\Phi,x,\tilde{g}q_\ell}(t,y).
\end{align*}
Clearly (P5) holds also for $\tilde{g}q_\ell$ and then, if $0<\varepsilon\leq 1$,
 we get from \theo{IestL2},
$$
  \sup_{t\in[0,T]}||I(t,\,\cdot\,)||_{L_2(\Real^n)}
  \leq 
  \sum_{|\ell|=0}^d
   \varepsilon^{\frac{n+|\gamma|+|\ell|}{2}}
 \sup_{t\in[0,T]}||
 \Int^{\gamma+\ell}_{\Phi,x,\tilde{g}q_\ell}(t,\,\cdot\,)||_{L_2(\Real^n)}
 \leq C(T)\varepsilon^{\frac{n+|\gamma|}{2}}.
$$
Therefore
\begin{align*}
 \lefteqn{ \sup_{t\in[0,T]}||
 \partial_y^\beta\Int^\alpha_{\Phi,x,g}(t,\,\cdot\,)||_{L_2(\Real^n)}}{\hskip 10 mm} &\\
&\leq C(T)
\left(
  \sum_{\beta_1+\beta_2=\beta}
  \sum_{|\gamma|=0}^{|\alpha|}
  \varepsilon^{-\frac{n+|\alpha|}{2}}
  \varepsilon^{|\alpha|-|\beta_2|-|\gamma|} 
\varepsilon^{\frac{n+|\gamma|}{2}}
+ e^{-w_4 \eta ^2/\varepsilon}\right)
\leq C(T)
  \varepsilon^{-|\beta|},
\end{align*}
for all $\varepsilon\in(0,1]$.
From this last estimate it immediately follows that also
$$
\sup_{t\in[0,T]}||\Int^\alpha_{\Phi,x,g}(t,\,\cdot\,)||_{H_\varepsilon^s(\Real^n)}
=\sup_{t\in[0,T]}
\sum_{|\beta|\leq s}
\varepsilon^{|\beta|-s}||\partial_y^\beta\Int^\alpha_{\Phi,x,g}(t,\,\cdot\,)||_{L^2(\Real^n)}
\leq C(T)\varepsilon^{-s}.
$$
Since when $0<\varepsilon\leq 1$, we clearly have
$||\,\cdot\,||_{H^s(\Real^n)}\leq
||\,\cdot\,||_{H_\varepsilon^s(\Real^n)}$
the theorem is proved.


\section{Error Estimates in Max Norm}\lbsec{Maxnorm}


We will here consider max norm estimates for Gaussian beams applied to
\eq{Schrodinger} and \eq{WaveEq}.
The main result is \theo{MaxNormError} in 
\sect{mainresult}. 
Also in the case of max norm estimates the 
oscillatory integrals in \eq{IntDef} play a crucial role.
However, here slightly different assumptions are made for the functions
in the integrals, and they are estimated pointwise. 
In \sect{prelim},
 we define notation and the sets used in \theo{MaxNormError}. 
The statement of the theorem and the general steps of the proof 
 are then given in \sect{mainresult}. 
 Finally, the details of these steps, in the form of two
 secondary theorems,
 are proved
 in \sect{DiffRepSec} and \sect{IestimateSec}.
 

\subsection{Preliminaries}\lbsec{prelim}

For the proof of the max norm estimates the assumptions (A1)--(A4)
must hold for a slightly larger set than $K_0$,
where
the initial amplitude is supported. We therefore
define the family of compact sets that ``fatten'' the set $K_0$,
$$
  K_{d} = \{ z\in \Real^n\ :\ {\rm dist}(z,K_0)\leq d\}\supset K_0.
$$
We also introduce the corresponding space-time set,
$$
{\mathcal K}_d = [0,T]\times K_d.
$$
Clearly (A1), (A2) and (A4) hold with $K_0$ replaced by $K_d$,
for any $d>0$. Since the initial phase $\varphi_0$ is smooth, 
we can also always find some, small enough,
$d$ such that (A3) holds. We will henceforth
consider a fixed such $d$. Then, all results in 
previous sections will be true, if $K_d$ is used instead of $K_0$.
Note that the cutoff width $\eta$ must now be
admissible for $K_d$ rather than $K_0$. The oscillatory
integrals can still be taken over $K_0$ though, since it contains
the support of the amplitude functions.

For the remaining definitions we recall that by \sect{GBexist} the
ray function
$x(t,z)$ is smooth under our assumptions.
We define the
 Jacobian $J$ by
$$
J(t,z) := D_z x(t,z).
$$
Furthermore, we introduce 
the set of caustic points on $[0,T]\times \Real^n$ for a central ray function $x(t,z)$,
$$
\mathcal{C}_x = \left\{(t,y)\in[0,T]\times \Real^n\ :\ 
\exists (t,z)\in {\mathcal K}_d\ \text{\rm such that\ } y=x(t,z),\
\det J(t,z) = 0\right\},
$$
and the fattened caustic set,
$$
\mathcal{C}_{x,\delta} = \left\{(t,y)\in [0,T]\times \Real^n\ :\ 
\text{\rm dist}((t,y),\mathcal{C}_x)< \delta\right\}.
$$
We also let ${\mathcal D}_{x,\delta}$ be the fattened domain of 
$x(t,z)$,
$$
  {\mathcal D}_{x,\delta} = \{ (t,y)\in [0,T]\times\Real^n\ :\ {\rm dist}(y,x(t,K_0))\leq \delta\}.
$$
Note that when $\varepsilon\to 0$ the solution will concentrate on the
set ${\mathcal D}_{x,0}$. Hence, ${\mathcal D}_{x,\delta}$ can be thought
of as approximating the essential support of the solution.
In Figure~\ref{OR:figure:Sets}, the sets are visualized for an example in two dimensions.

\begin{figure}[!t]
  \captionsetup[subfigure]{labelformat=empty}
\centerline{
\subfloat[$\mathcal{K}_0$]{\includegraphics[width=.46\textwidth]{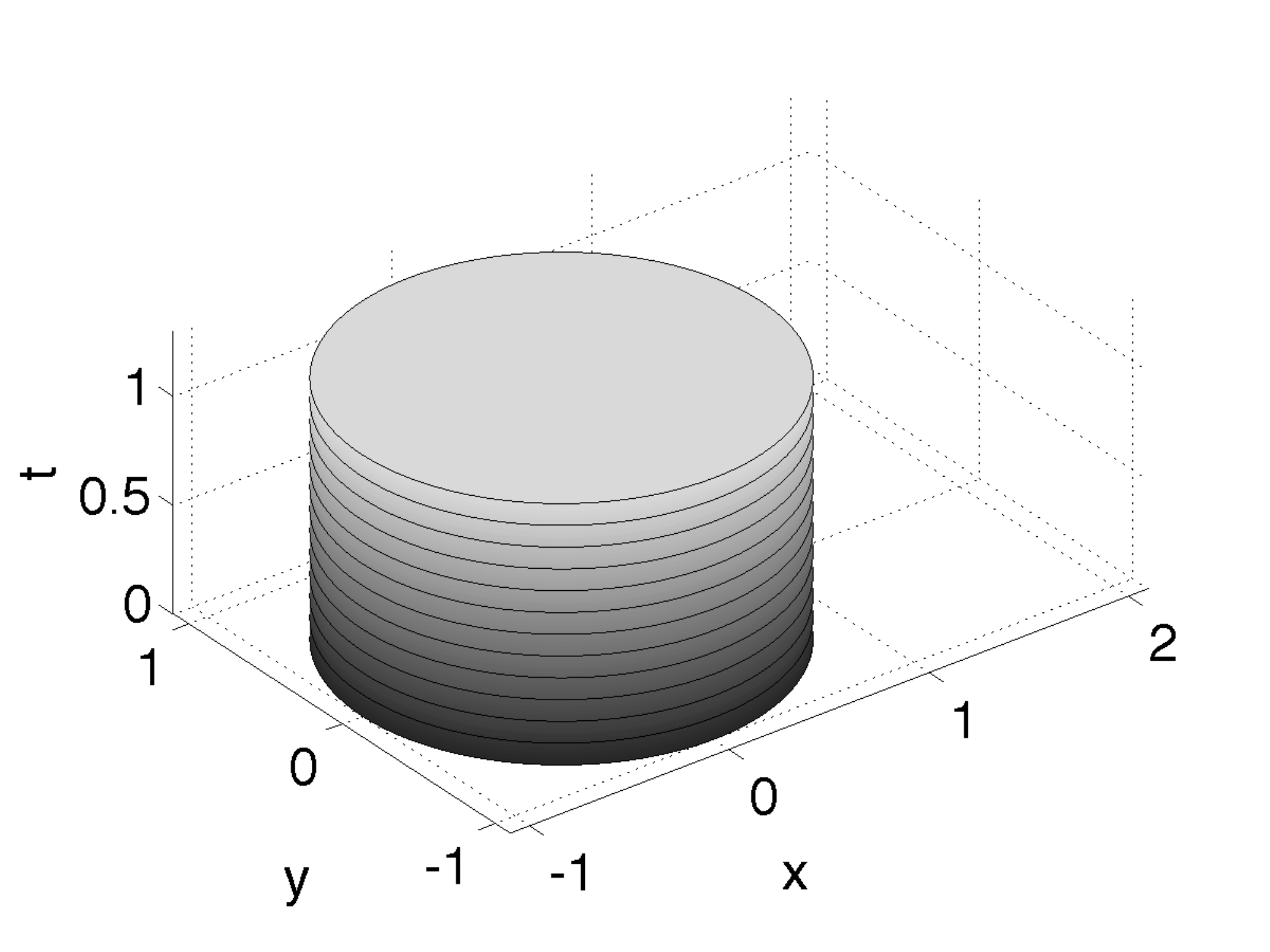}}
\subfloat[$\mathcal{D}_{x,0}$]{\includegraphics[width=.46\textwidth]{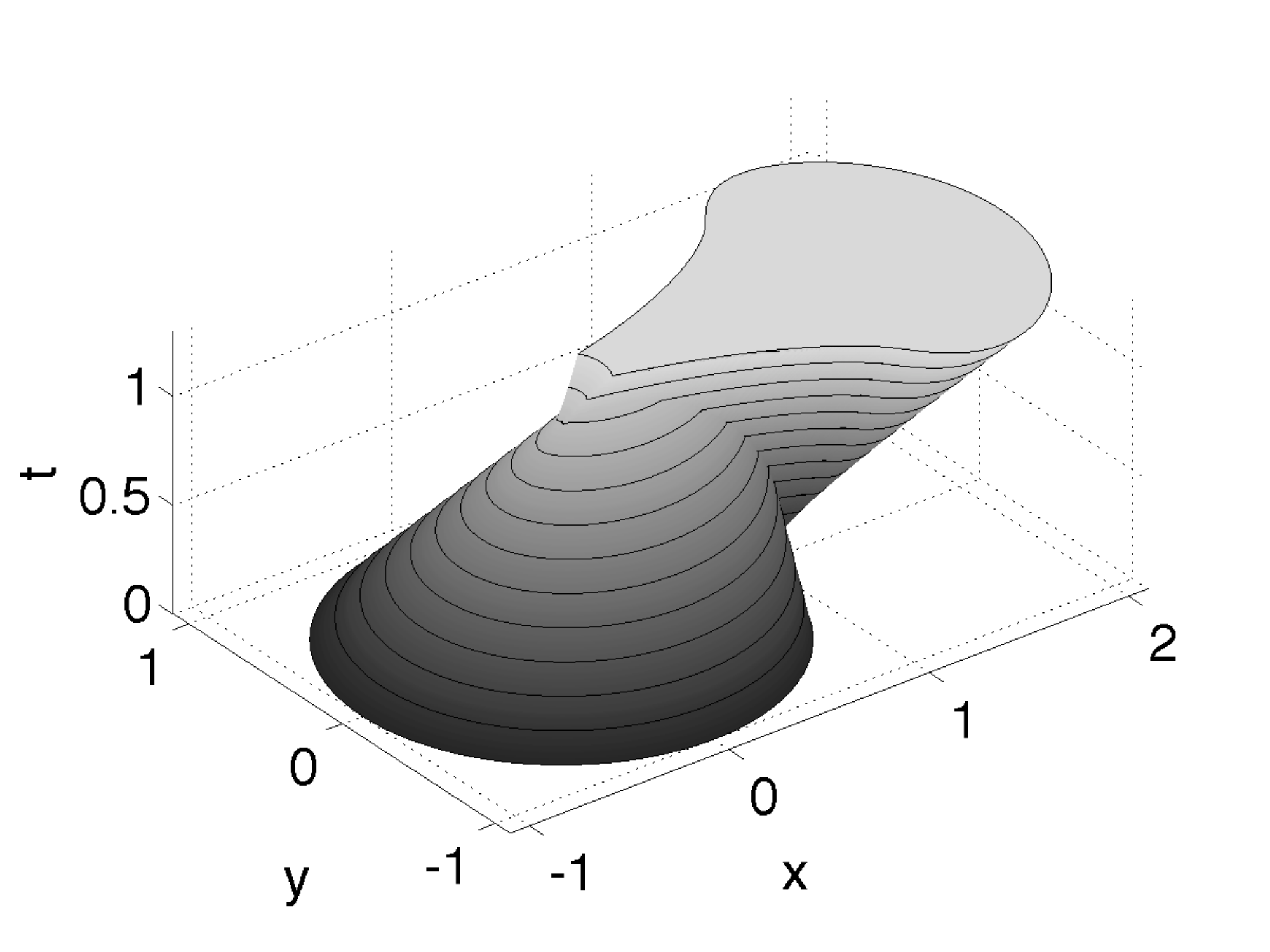}}
}
\centerline{
\subfloat[$X^{-1}(\mathcal{C}_x)$ and $K_0$]{\includegraphics[width=.46\textwidth]{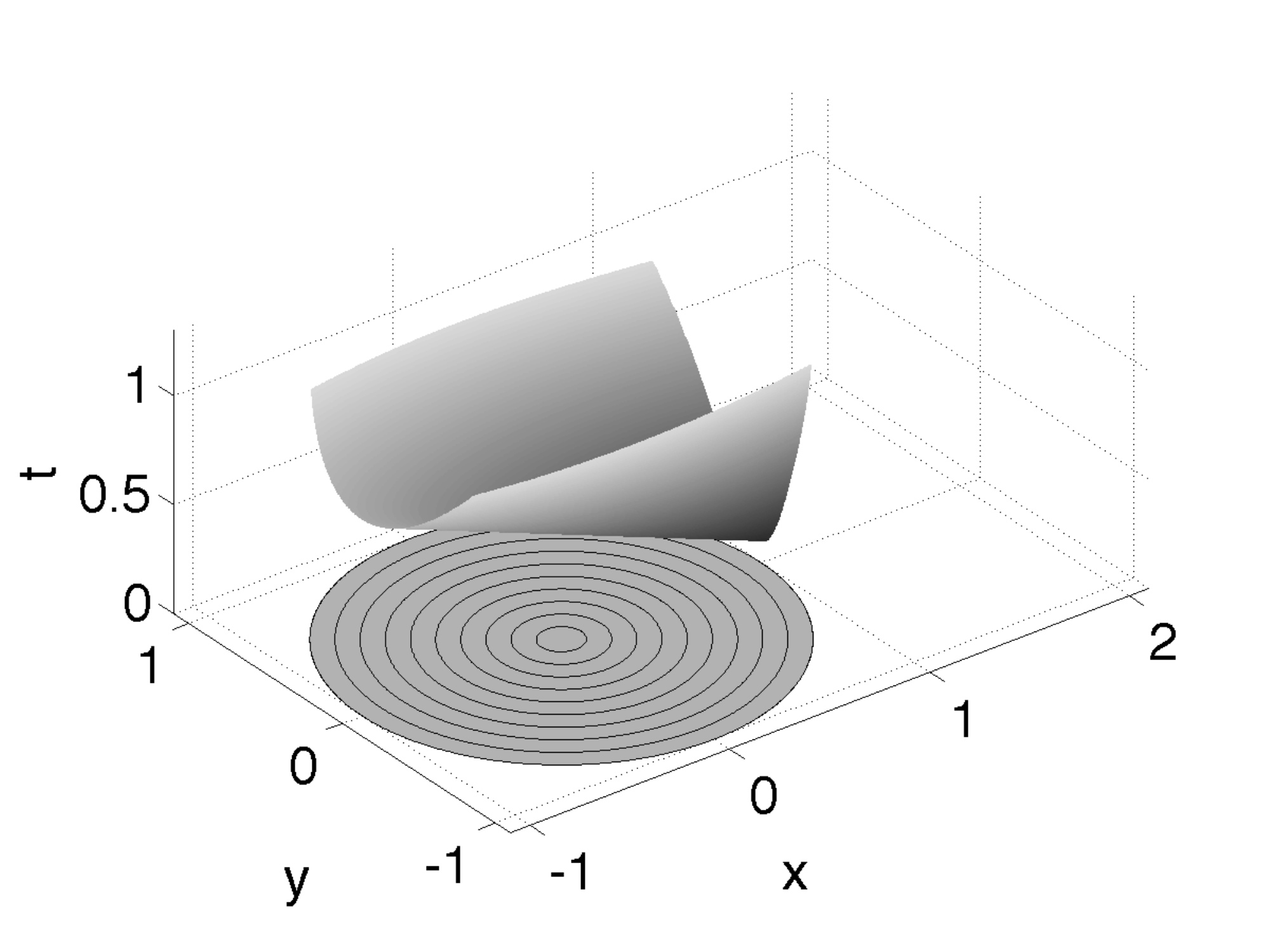}}
\subfloat[$\mathcal{C}_x$ and $K_0$]{\includegraphics[width=.46\textwidth]{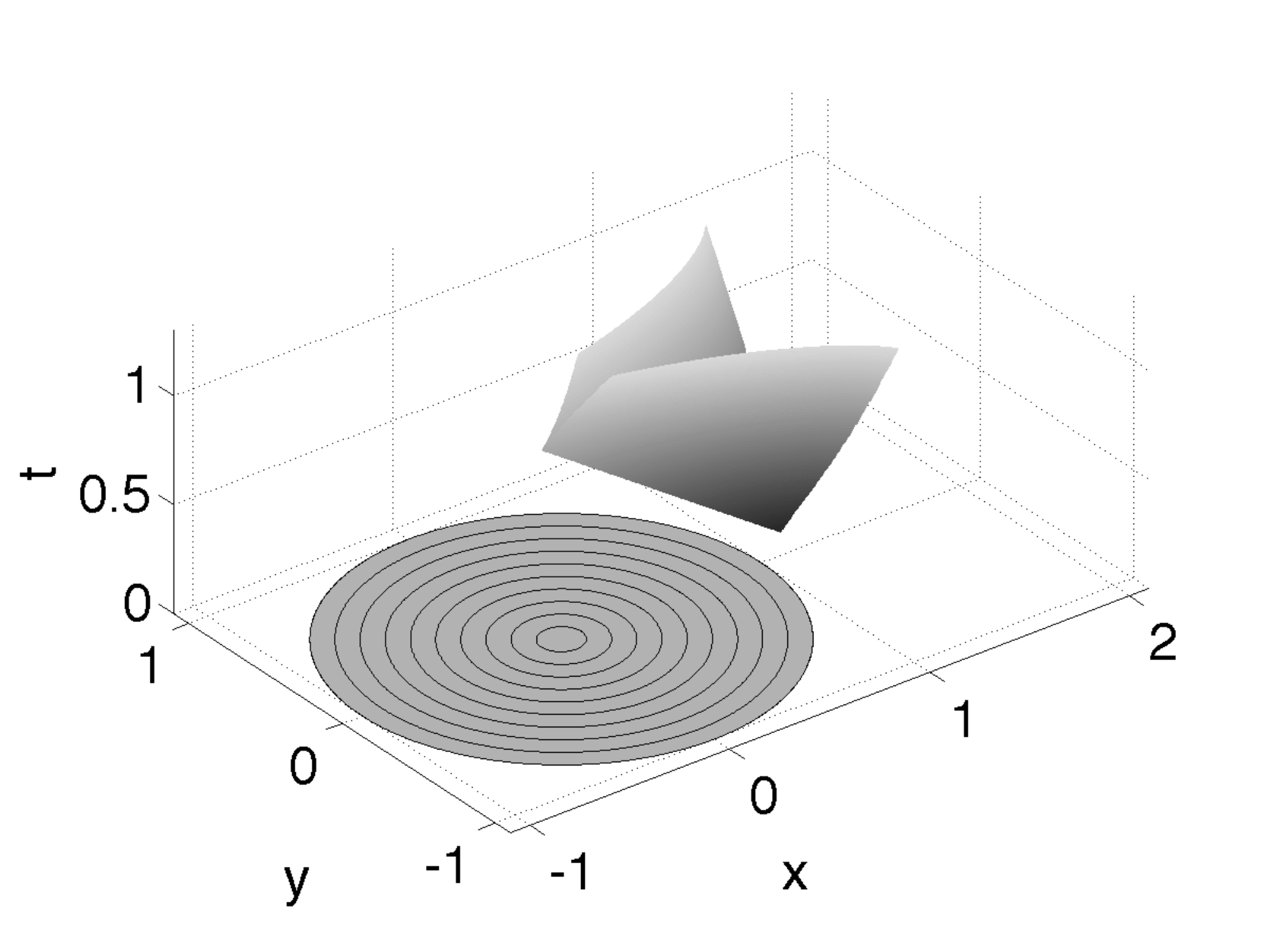}}
}
\centerline{
\subfloat[$\mathcal{K}_0$ and $X^{-1}(\mathcal{C}_x)$ at $t=0.8$]{\includegraphics[width=.46\textwidth]{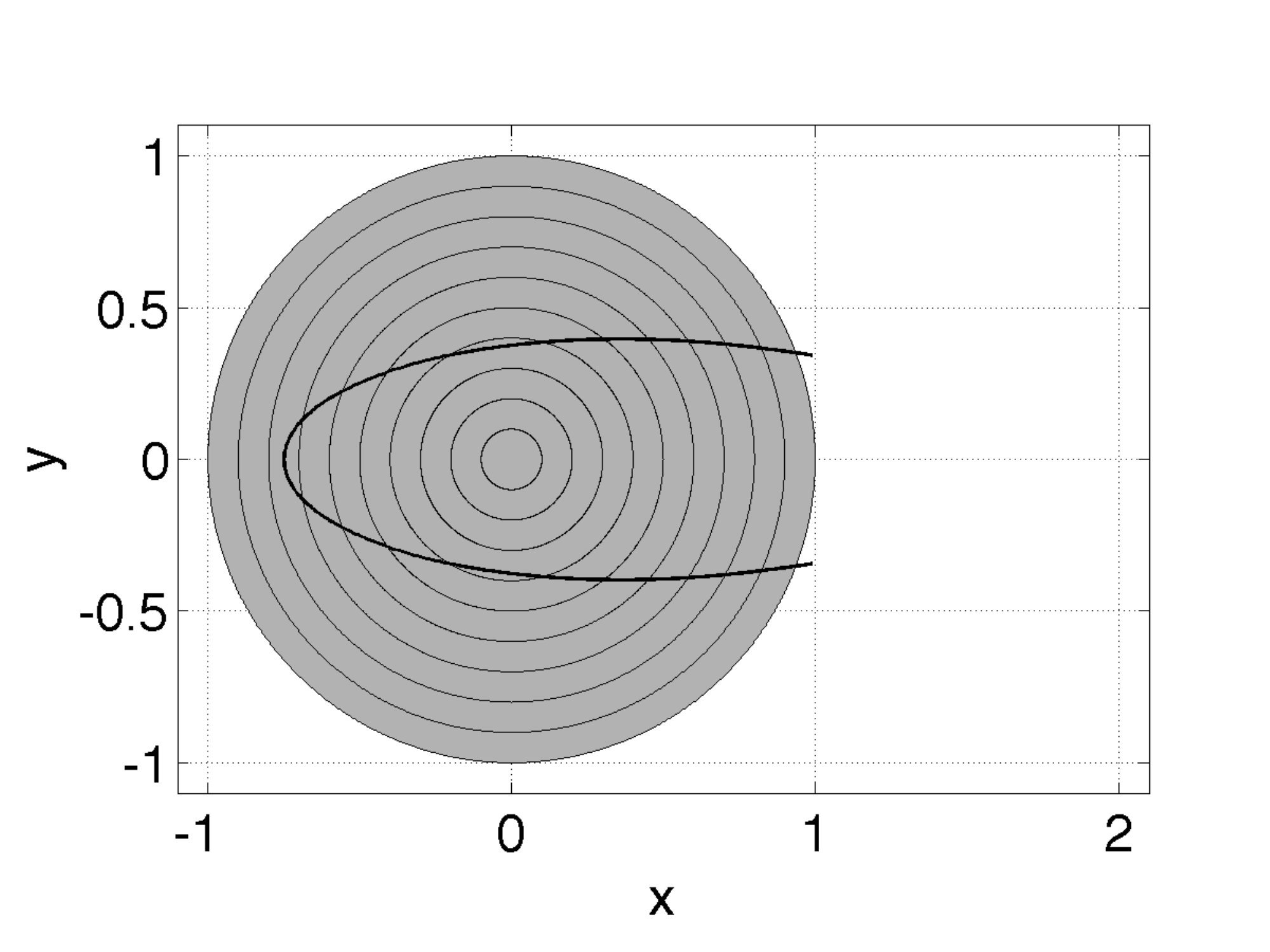}}
\subfloat[$\mathcal{D}_{x,0}$
and $\mathcal{C}_x$ at $t=0.8$]{\includegraphics[width=.46\textwidth]{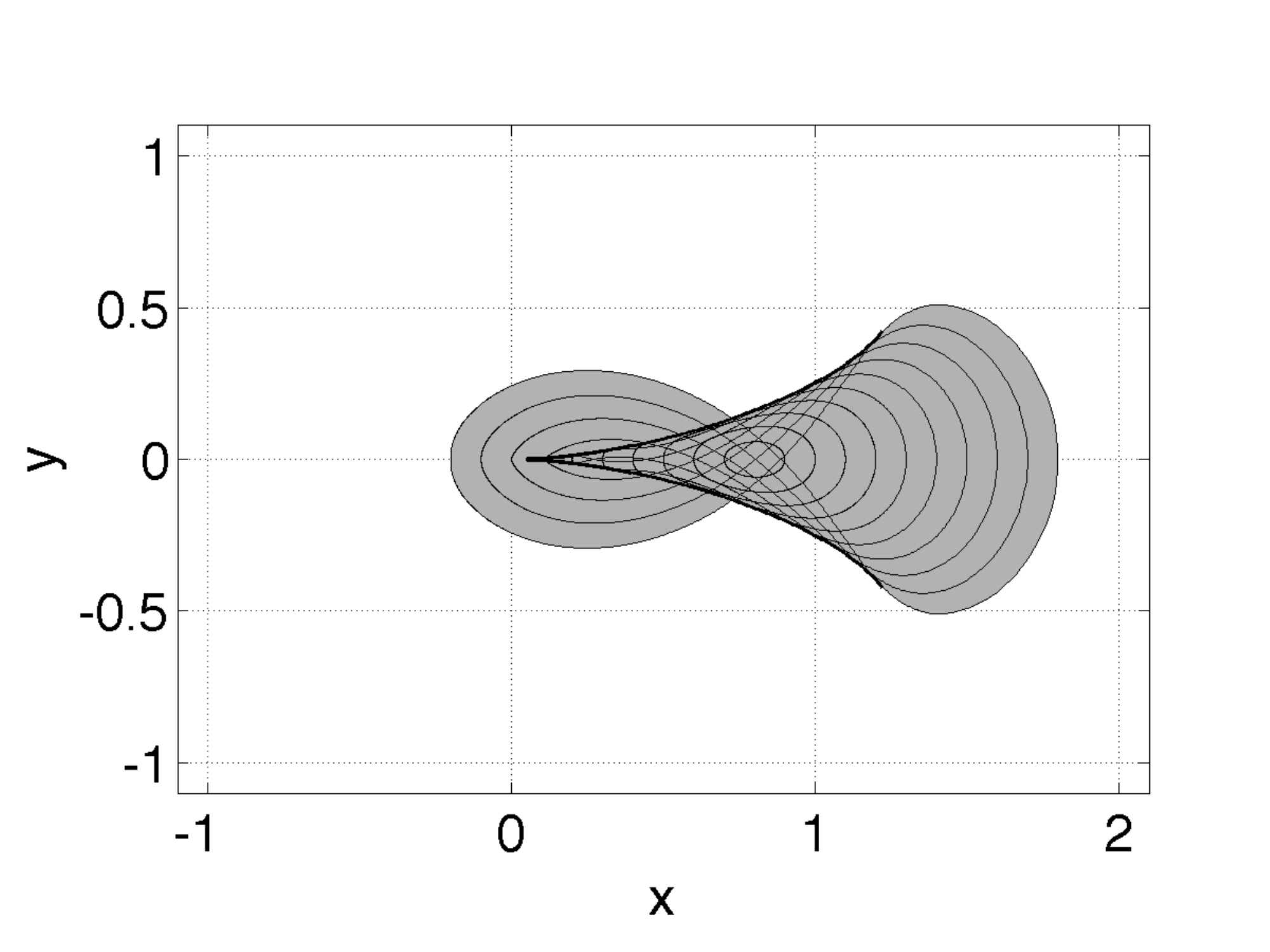}}
}
\caption{Examples of the 
the various sets used in this section for a two-dimensional case,
where $\varphi_0(x,y)=-x+y^2+0.4\,x^2$, $T=1.2$ and $K_0$ is the 
unit circle. In the last row the intersection of the sets with
the plane $t=0.8$ is shown; the solid black line indicates 
$X^{-1}({\mathcal C}_x)$ and ${\mathcal C}_x$ respectively.
}
\label{OR:figure:Sets}
\end{figure}

The total caustic set $\mathcal{C}_\delta$ 
and domain $\mathcal{D}_\delta$ 
are finally defined as the union of the corresponding
sets of each mode, 
$$
  \mathcal{C}_\delta = \begin{cases}
  \mathcal{C}_{x,\delta}, & \text{Schr\"odinger},\\
  \mathcal{C}_{x^+,\delta}\cup\mathcal{C}_{x^-,\delta}, & \text{wave equation},
  \end{cases}
  \qquad
  \mathcal{D}_\delta = \begin{cases}
  \mathcal{D}_{x,\delta}, & \text{Schr\"odinger},\\
  \mathcal{D}_{x^+,\delta}\cup\mathcal{D}_{x^-,\delta}, & \text{wave equation}.
  \end{cases}
$$
Note that for the wave equation an equivalent definition of $\mathcal{C}_\delta$
is the $\delta$-fattened version of 
$\mathcal{C}_{x^+}\cup\mathcal{C}_{x^-}$.
Moreover, we always consider $[0,T]\times \Real^n$ to be the
universal set and complements of sets are taken with respective to this, i.e.
for ${\mathcal U}\subset [0,T]\times \Real^n$,
$$
  {\mathcal U}^c = [0,T]\times \Real^n \setminus {\mathcal U}.
$$

Finally, in the proofs we will typically not use property (P4) the way it is written 
in \lem{GBphaseOKforQ},
but rather the following simple consequence, which we denote (P$4'$),
\begin{itemize}
\item[\bf (P${\mathbf 4'}${\bf )}] {\it there exists a constant $w_4>0$ such that}
$$
\left|e^{i\Phi(t,y,z)/\varepsilon}\varrho_\eta(y)\right|\leq e^{-w_4|y|^2/\varepsilon},
$$
{\it for all $(t,z)\in {\mathcal K}_d$ and $y\in\Real^n$.}
\end{itemize}

\begin{remark}
Note that
the caustic set is fattened both in space and time. 
This is necessary for the estimates derived below to be true;
the rate $\varepsilon^{\lceil k/2\rceil}$ is only obtained
uniformly away from the caustics, in space and time.
\end{remark}

\subsection{Main Result}\lbsec{mainresult}

We are now ready to state the main theorem of this section.
It gives max norm error estimates in terms of $\varepsilon$,
over different parts of the
solution domain. 
The theorem shows that uniformly away from caustics, $(t,y)\in{\mathcal C}^c_\delta$,
the convergence rate is the same $O(\varepsilon^{k/2})$ 
as in \cite{LiuRunborgTanushev:10} when $k$ is even. 
For odd $k$, however,
error cancellations between adjacent beams can be exploited, and
the better rate $O(\varepsilon^{(k+1)/2})$ 
is obtained,
similar to
the results in \cite{Zhen:14,MotamedRunborg:09}. We believe this rate
is sharp.
Close to a caustic point,
$(t,y)\in {\mathcal C}_{\delta}$, the theorem gives the
rather coarse rate estimate $O(\varepsilon^{(k-n)/2})$, which
can likely be improved for many types of caustics.
Finally, away from the essential support of the
solution,  $(t,y)\in {\mathcal D}_{\delta}^c$, 
the convergence is exponential in $\varepsilon$.
In fact, the solution itself is also exponentially
small in $\varepsilon$ on this domain.

\begin{theorem}\lbtheo{MaxNormError}
Let $u_k$ be
the $k$-th order Gaussian beam
superposition given in \sect{GBdef} for the Schr\"odinger
equation \eq{Schrodinger} or the wave equation \eq{WaveEq}, with a cutoff width
$\eta$ that is admissible for $K_d$, $T>0$ and the 
correspondning Gaussian beam phases,
$\Phi_k$ or $\Phi^\pm_k$.
If $u$ is the exact solution to Schr\"odinger's equation
or the wave equation, then we have the following
estimate. For each $\delta>0$ and $m>0$, there is a constant $C_{\delta,m}$
such that
\be{maininfres}
 |u_k(t,y)-u(t,y)| \leq C_{\delta,m}
     \begin{cases}
     \varepsilon^{\lceil k/2\rceil}, & (t,y)\in  {\mathcal C}^c_\delta, \\
   \varepsilon^{(k-n)/2}, & (t,y)\in {\mathcal C}_\delta,\\
     \varepsilon^{m}, & (t,y)\in  {\mathcal D}^c_\delta,
   \end{cases}   
   \qquad \forall  \varepsilon\in (0,1].
\ee
\end{theorem}

The theorem also immediately gives us an estimate for the initial data
in all $L_p$-norms.
\begin{corollary}
Under the same conditions as in \theo{MaxNormError}, 
there is a constant $C_p$ 
for each $1\leq p\leq \infty$ 
such that
\be{initres}
||u_k(0,y)-u(0,y)||_{L^p(\Real^n)} \leq C_p
     \varepsilon^{\lceil k/2\rceil},\qquad 1\leq p \leq \infty,
     \quad
     \forall \varepsilon\in(0,1].
\ee
\end{corollary}
\begin{proof}
Since $x(0,z)=z$ and $K_d$ is compact,  there exists $\delta>0$ such that
 $\det J(t,z)\neq 0$ for $t\in[0,\delta]$ and $z\in K_d$.
Hence,   there is a caustic free initial interval $[0,\delta]$ and
for $T=\delta$, the fattened caustic
set ${\mathcal C}_{\delta}$ is empty.
\theo{MaxNormError} then shows that there is a constant $C$ such that
for all $\varepsilon\in(0,1]$.
$$
 |u_k(t,y)-u(t,y)| \leq C
     \varepsilon^{\lceil k/2\rceil}, \qquad  \forall (t,y)\in [0,\delta]\times\Real^n.
$$
Since initial data for both $u_k$ and $u$ is compactly supported,
the result extends to all $L_p$-norms at $t=0$.
\end{proof}

We prove \theo{MaxNormError} starting from
a standard Sobolev inequality and the result in the previous section, 
namely
\be{sobolevpessimistic}
\sup_{t\in[0,T]}  ||u(t,\,\cdot\,)-u_k(t,\,\cdot\,)||_{L^\infty(\Real^n)} 
  \leq C\sup_{t\in[0,T]} ||u(t,\,\cdot\,)-u_k(t,\,\cdot\,)||_{H^s(\Real^n)} \leq 
  C \varepsilon^{\frac{k}{2}-s},
\ee
for any $s>n/2$, and $s\geq 1$ for the wave equation. 
We take $s=\lfloor n/2\rfloor +1$ to ensure this.
The estimate \eq{sobolevpessimistic} is rather pessimistic.
However, we can improve it by using the fact that better estimates
can be proved for the difference between beams of different orders.
Let $p=2\lfloor n/2\rfloor+3+m'=2s+1+m'$
where $m'\in\Znumbers^+$ and $m'\geq \max(2m-k-1,0)$.
Assume that 
$\eta$ is admissible also for $K_d$, $T$ and the 
higher order Gaussian beam phase $\Phi_{k+p}$, for the Schr\"odinger equation, or
$\Phi^\pm_{k+p}$ for the wave equation.
Then, by \eq{sobolevpessimistic}
\begin{align}\lbeq{ekest}
   |u(t,y)-u_k(t,y)|   &\leq ||u(t,\,\cdot\,)-u_{k+p}(t,\,\cdot\,)||_{L^\infty(\Real^n)}   + |u_{k+p}(t,y)-u_k(t,y)|  \nonumber\\
   &\leq C\varepsilon^{(k+p)/2-s} 
   +|u_{k+p}(t,y)-u_k(t,y)|,
\end{align}
for $(t,y)\in[0,T]\times\Real^n$.
We now need to use a representation result similiar to
\lem{PuInTermsOfQ} showing that
the difference between beams of different orders can
be written as a sum of oscillatory integrals of the type \eq{IntDef},
but where the property (P5) is replaced by three new properties, namely:

\begin{itemize}
\item[\bf (P6)] $\Phi(t,0,z)$ and
$\nabla_y\Phi(t,0,z)$ are real and 
\be{B3}
  J(t,z)^T\nabla_y\Phi(t,0,z)=\nabla_z\Phi(t,0,z),
\ee
for all $t\geq 0$ and $z\in\Real^n$.
\item[\bf (P7)]
 $g(t,y,z,\varepsilon)\in L^\infty([0,T]\times\Real^{n}\times{K}_d\times \Real^+)$
is compactly supported in $K_0$ for fixed $(t,y,\varepsilon)$, and 
there are positive constants $D_7$, $w_7$, such that
for all $(t,z)\in{\mathcal K}_d$, $\varepsilon>0$ and $y\in\Real^n$,
\be{B4}
    \left|g(t,y,z,\varepsilon)e^{i\Phi(t,y-x(t,z),z)/\varepsilon}\varrho_\eta(y-x(t,z))\right|\leq 
    D_7e^{-w_7|y-x(t,z)|^2/\varepsilon},
    \ee
    \item[\bf (P8)] when $y_0=x(t,z_0)$,
there are positive constants $D_8$, $w_8$, such that
for all $t\in[0,T]$, $z,z_0\in K_d$,
$\varepsilon>0$ and $y_0\in \Real^n$,
\begin{align}\lbeq{B5}
\lefteqn{    \left|\Bigl(g(t,y_0,z,\varepsilon)-g(t,y_0,z_0,\varepsilon)\Bigr)
    e^{i\Phi(t,y_0-x(t,z),z)/\varepsilon}\varrho_\eta(y_0-x(t,z))
    \right|}
    \hskip 20 mm &\nonumber \\
    &\leq  D_8|z-z_0|\left(1+\frac{|z-z_0|^{q}}{\varepsilon^\ell}\right)e^{-w_8|y_0-x(t,z)|^2/\varepsilon},
\end{align}
with $q\geq 2\ell$.
\end{itemize}

We are then able to prove the following theorem.
\begin{theorem}\lbtheo{DiffRep}
Let $u_k$ and $u_{k+p}$ be
the $k$-th and $(k+p)$-th order Gaussian beam
superpositions given in \sect{GBdef} for the Schr\"odinger
equation \eq{Schrodinger} or the wave equation \eq{WaveEq}.
Suppose the same cutoff width $\eta$
 is used for both
$u_k$ and $u_{k+p}$.
Then
there is a finite $J$ such that
\be{DiffRep}
u_{k+p}(t,y)-u_k(t,y)=
\varepsilon^{\frac{k}{2}}\sum_{j=0}^{J} \varepsilon^{\ell_j} 
\Int^{\beta_j}_{\Psi_{j},x_j,g_j}(t,y),
\ee
where $(\Psi_j,x_j)$ is one of 
$(\Phi_k,x)$,
$(\Phi_{k+p},x)$, for the Schr\"odinger equation, or
$(\Phi^\pm_k,x^\pm)$,
$(\Phi^\pm_{k+p},x^\pm)$, for the wave equation. 
Moreover,
$\ell_j\geq 0$ and 
when $\ell_j=0$, the parity (odd/even) of
$|\beta_j|$ is the same as that of
$k$.

In addition, if
$\eta$ is admissible for $K_d$, $T>0$ and the 
corresponding Gaussian beam phases,
$\Phi_k$, $\Phi_{k+p}$,
for the Schr\"odinger equation, or
$\Phi^\pm_k$, $\Phi^\pm_{k+p}$,
for the wave equation, 
then each triplet $(\Psi_j,x_j,g_j)$ have properties (P1)--(P4) and (P6)--(P8).
\end{theorem}

Applying \theo{DiffRep} to \eq{ekest} yields for $t\in[0,T]$,
\be{ekest2}
   |u(t,y)-u_k(t,y)|
\leq C\varepsilon^{(k+1+m')/2} +
\varepsilon^{\frac{k}{2}} \sum_{j=0}^{J}\varepsilon^{\ell_j} 
\left|\Int^{\beta_j}_{\Psi_{j},x_j,g_j}(t,y)\right|,
\ee
where we used the fact that $(k+p)/2-s = (k+1+m')/2$.
The last piece needed to prove \theo{MaxNormError} is
a pointwise estimate of $\Int^{\alpha}_{\Phi,x,g}(t,y)$,
which is contained in the final theorem of this section,
\begin{theorem}\lbtheo{Iestimate}
If $(\Phi,x,g)$ have properties
(P1)--(P4) and (P6)--(P8), then, for each $\delta>0$
there are constants $C_\delta$ and $w_\delta>0$
such that
\be{Iestmax}
 \left|\Int^\alpha_{\Phi,x,g}(t,y)\right| \leq C_\delta
    \begin{cases}
     1, & \text{\rm $|\alpha|$ even, $(t,y)\in  {\mathcal C}^c_{x,\delta}$}, \\
   \varepsilon^{1/2}, & \text{\rm $|\alpha|$ odd, $(t,y)\in {\mathcal C}^c_{x,\delta}$}, \\
   \varepsilon^{-n/2}, & (t,y)\in {\mathcal C}_{x,\delta},\\
     \exp(-w_\delta/\varepsilon), & (t,y)\in  {\mathcal D}^c_{x,\delta},
   \end{cases}   
\ee
for all
$\varepsilon\in(0,1]$. The constants
$C_\delta$, $w_\delta$ depend on 
$\alpha$, $\Phi$, $x$, $g$.
\end{theorem}

Using \theo{Iestimate} in \eq{ekest2}
we have for $(t,y)\in  {\mathcal C}^c_\delta\subset (\cup_j {\mathcal C}_{x_j,\delta})^c=
\cap_j {\mathcal C}^c_{x_j,\delta}$,
$$
\varepsilon^{\ell_j} 
\left|\Int^{\beta_j}_{\Psi_{j},x_j,g_j}(t,y)\right|
\leq C \begin{cases}
1, & \text{$\ell_j=0$ and $k$ even},\\
\varepsilon^{1/2},& \text{$\ell_j=0$ and $k$ odd},\\
\varepsilon^{\ell_j}, & \ell_j\geq 1,
\end{cases}
\leq
C \begin{cases}
1, & \text{$k$ even},\\
\varepsilon^{1/2},& \text{$k$ odd},
\end{cases}
$$
since $k$ and $|\beta_j|$ have the same parity when $\ell_j=0$
and $\varepsilon\in(0,1]$.
Therefore, 
$$
\varepsilon^{\frac{k}{2}}\varepsilon^{\ell_j} 
\left|\Int^{\beta_j}_{\Psi_{j},x_j,g_j}(t,y)\right|
\leq C \varepsilon^{\lceil k/2\rceil},
$$
and because $m'\geq 0$, the first case in \eq{maininfres}
is proved.
When
$(t,y)\in  {\mathcal D}^c_\delta\subset (\cup_j {\mathcal D}_{x_j,\delta})^c=
\cap_j {\mathcal D}^c_{x_j,\delta}$,
the second term in \eq{ekest2} is 
asymptotically smaller than all powers of $\varepsilon$, so the first
term in \eq{ekest2}
dominates, irrespective of $m'\geq 0$. This shows the third
case in \eq{maininfres} since $(k+1+m')/2\geq m$. The second case is finally estimated simply
by the largest term in \theo{Iestimate}. 
\theo{MaxNormError} is thereby proved,
if $\eta$ is indeed admissible for the higher order phase
$\Phi_{k+p}$ or $\Phi^\pm_{k+p}$. 
If not, let $\tilde{\eta}<\eta$ be an admissible cutoff width
for $K_d$, $T$ and the higher order phase. 
\lem{GBphaseOKforQ} ensures the existence of 
such $\tilde{\eta}$.
Denote by $\tilde{u}_k$ and $\tilde{u}_{k+p}$ the Gaussian beam
superpositions of orders $k$ and $k+p$ respectively, which (both)
use $\tilde\eta$ as cutoff width. This width is clearly
admissible for both of them and therefore the theorem holds for $\tilde{u}_k$.
Moreover, by \lem{cutoffdiff}, the difference $|u_k-\tilde{u}_k|$
is exponentially small in $\varepsilon$, which implies that the theorem
also holds for $u_k$.

The remainder of this section is dedicated to proving \theo{DiffRep}
and \theo{Iestimate}. 


\subsection{Proof of \theo{DiffRep}}\lbsec{DiffRepSec}


As we will show below, the Gaussian beam phase $\Psi_j$ of the
oscillatory integrals in \eq{DiffRep} is always one of
$\Phi_k$, $\Phi_{k+p}$, for the Schr\"odinger equation,
and one of
$\Phi^\pm_k$, $\Phi^\pm_{k+p}$, for the wave equation.
All these phases, and their corresponding central rays $x$, $x^\pm$,
have properties (P1)--(P4) by
\lem{GBphaseOKforQ}, and the assumption on $\eta$.
The first step in the proof is a lemma proving that these
phases 
also satisfy (P6).
\begin{lemma}\lblem{GBphase}
For all $k\geq 0$, property (P6) is true for
the Schr\"odinger phase $\Phi_k$ and its central ray $x$,
as well as for the phases
$\Phi_k^\pm$ and central rays $x^\pm$ of the wave equation.
\end{lemma}
\begin{proof}
As noted in \rem{hamiltonian}, the first three equations
in \eq{GBODESch} and \eq{GBODEWave} have the Hamiltonian
structure in \eq{GBODEGen}. Let $\phi$ and $H$ represent the phase
and Hamiltonian
for the Schr\"odinger equation or
one of the modes of the wave equation.
Moreover, let
$\phi_0$, $x$ and $p$ be the corresponding
phase, central ray and ray direction. They
are well-defined for all $t\geq 0$
and $z\in \Real^n$
by the discussion in \sect{GBexist}.  They are also real, since
the initial data \eq{GBini} is real and
$H(x,p)$ is real whenever $x$
and $p$ are real.
The first part of (P6) is then proved by noting that
$\phi(t,0,z)=\phi_0(t,z)$ and
$\nabla\phi(t,0,z)=p(t,z)$.
Next, let $J(t,z)=D_z x(t,z)$ and define
$$
S(t,z):=
J(t,z)^T\nabla_y\phi(t,0,z)-\nabla_z\phi(t,0,z)
= J(t,z)^Tp(t,z) - \nabla_z\phi_0(t,z),
$$
which is zero at $t=0$ by \eq{GBini}.
From \eq{GBODEGen}, with $P(t,z)=D_z p(t,z)$, it then follows that
\begin{align*}
\partial_t S
&= (D_z\partial_{t} x)^Tp + J^T\partial_tp - \nabla_z\partial_{t}\phi_0\\
&= (D_{z}\nabla_pH)^Tp - J^T\nabla_y H - \nabla_{z}
(-H+(\nabla_p H)^T p)\\
&= (D_{z}\partial_pH)^Tp - J^T\nabla_y H 
+J^T\nabla_y H+
P^T\nabla_p H
- (D_{z}\nabla_p H)^T p - P^T\nabla_p H
= 0.
\end{align*}
This shows that $S$ is zero for all times, which proves the lemma.
\end{proof}

We will now continue with the proof for the Schr\"odinger case.
Since the wave equation beams are just sums of beams for its
two modes, the proof for the wave equation case will be identical,
and we leave it out.

By \eq{udef} we have for the Schr\"odinger equation
$$
u_{k+p}(t,y)-u_k(t,y) =
    \left(\frac{1}{2\pi\varepsilon}\right)^\frac{n}{2} 
   \int_{K_0} [v_{k+p}(t,y,z)-v_{k}(t,y,z)] \varrho_\eta(y-x(t,z)) dz,
$$
since the same $\eta$ is used for
the $k$-th and the $(k+p)$-th order beams.

Starting from the expressions for $\Phi_{k}$ and $A_{k}$ in
\eq{phidef} and \eqtwo{Adef1}{Adef2} we can analyze the 
differences $v_{k+p}-v_{k}$. 
We obtain
\begin{align}\lbeq{vsplit}
 v_{k+p}-v_k &= A_{k+p}e^{i\Phi_{k+p}/\varepsilon}
 -A_{k}e^{i\Phi_{k}/\varepsilon}\nonumber\\
 &=
 \left(A_{k+p}
 -A_{k}\right)e^{i\Phi_{k+p}/\varepsilon}
 +A_k\left(e^{i\Phi_{k+p}/\varepsilon}
 -e^{i\Phi_{k}/\varepsilon}\right).
\end{align}
By the discussion in \sect{GBexist}
none of $x$, $p$, $\phi_0$, $M$, $\phi_\beta$ or $a_{j,\beta}$
depend on $k$. Therefore,
\begin{align*}
A_{k+p}(t,y,z)-A_k(t,y,z)
  &=
  \sum_{j=0}^{\lceil \frac{k}{2} \rceil -1} \varepsilon^j \left[\bar{a}_{j,k+p}(t,y,z)
  -
  \bar{a}_{j,k}(t,y,z)\right]
  \\
  &\mbox{\ \ \ }
  +\sum_{j=\lceil \frac{k}{2} \rceil}^{\lceil (k+p)/2 \rceil -1} \varepsilon^j \bar{a}_{j,k+p}(t,y,z)\\
  &=
  \sum_{j=0}^{\lceil \frac{k}{2} \rceil -1} 
  \sum_{|\beta|=k-2j}^{k+p-2j-1} \frac1{\beta!}a_{j,\beta}(t,z) \varepsilon^jy^\beta
  \\
&  \mbox{\ \ \ }
  +\sum_{j=\lceil \frac{k}{2} \rceil}^{\lceil (k+p)/2 \rceil -1}
  \sum_{|\beta|=0}^{k+p-2j-1} \frac1{\beta!}a_{j,\beta}(t,z) \varepsilon^jy^\beta.
\end{align*}
This is a finite sum of terms having the form $a_{j,\beta}(t,z) \varepsilon^jy^\beta/\beta!$. It can easily be checked that $j+|\beta|/2\geq \frac{k}{2}$
for all terms.
Therefore, for some finite $N_a$, functions $g_j$, multi-indices $\alpha_j$
and powers $\ell_j\geq 0$,
we can write the sum as
$$
  A_{k+p}(t,y,z)-A_k(t,y,z) = \varepsilon^{\frac{k}{2}}\sum_{j=0}^{N_a}  
  \varepsilon^{\ell_j-|\alpha_j|/2}g_{j}(t,z)y^{\alpha_j},
$$
where the $g_j$ functions are equal to scaled amplitude coefficients,
which satisfy \eq{coeffsmooth}
and \eq{acompact}.
Moreover, if $\ell_j=0$ then $|\alpha_j| = k-2j$, so $|\alpha_j|$ then has
the same parity as $k$.
In \eq{vsplit} the amplitudes and phases are
evaluated at $y-x(t,z)$ and
hence, the first term there
contributes to $u_{k+p}-u_k$ as
\be{Aops}
    \left(\frac{1}{2\pi\varepsilon}\right)^\frac{n}{2} 
   \int_{K_0} \left(A_{k+p}-A_{k}\right)e^{i\Phi_{k+p}/\varepsilon}
  \varrho_\eta dz
  = \varepsilon^{\frac{k}{2}}
   \sum_{j=0}^{N_a}\varepsilon^{\ell_j} \Int^{\alpha_j}_{\Phi_{k+p},x,g_j},
\ee
where $|\alpha_j|$ has the same parity as $k$ when $\ell_j=0$.
For this case the $g_j$ functions 
are inde\-pendent of both $y$ and $\varepsilon$,
and by \eq{acompact} they have supp $g_j\subset K_0$.
Therefore, by \eq{coeffsmooth} and \eq{suplipfin}, property (P$4'$) 
implies (P7) and (P8), with $w_7=w_8=w_4$ and
$$
D_7 = \sup_{t\in[0,T]}||g_j(t,\,\cdot\,)||_{L^\infty(K_d)},\quad
D_8=\sup_{t\in[0,T]}|g_j(t,\,\cdot\,)|_{{\rm Lip}(K_d)},\quad
q=\ell=0.
$$
We conclude that the oscillatory integrals in \eq{Aops} all
satisfy (P1)--(P4) and (P6)--(P8).

We now consider the second term in \eq{vsplit} and define the
function
\be{gtilde}
\tilde{g}(t,y,z,\varepsilon) := \int_0^1
   e^{is(\Phi_{k+p}(t,y,z)-\Phi_k(t,y,z))/\varepsilon} ds. 
\ee
By \eq{coeffsmooth} we have
$\tilde{g}(t,y,z,\varepsilon)\in C^{\infty}
([0,T]\times\Real^{n}\times{K}_d\times \Real^+)$. 
A simple calculation shows that
$$
  e^{i\Phi_{k+p}/\varepsilon}-e^{i\Phi_k/\varepsilon} = 
  \left(e^{i(\Phi_{k+p}-\Phi_k)/\varepsilon}-1\right)e^{i\Phi_k/\varepsilon} = 
  \frac{i}{\varepsilon}\tilde{g}(\Phi_{k+p}-\Phi_{k})
  e^{i\Phi_{k}/\varepsilon}. 
$$
Then we have
\begin{align*}
\lefteqn{ A_k(t,y,z)\left(e^{i\Phi_{k+p}(t,y,z)/\varepsilon}
 -e^{i\Phi_{k}(t,y,z)/\varepsilon}\right)
 }{\hskip 1 cm}\\
 &=
 \frac{i}{\varepsilon}\tilde{g}(t,y,z,\varepsilon)A_k(t,y,z)
 \Bigl(
 \Phi_{k+p}(t,y,z)-\Phi_{k}(t,y,z)\Bigr) 
 e^{i\Phi_{k}(t,y,z)/\varepsilon}
 \\
 &=
 \frac{i}{\varepsilon}
 \tilde{g}(t,y,z,\varepsilon)A_k(t,y,z)
 \sum_{|\beta|=k+2}^{k+p+1} \frac1{\beta!}\phi_{\beta}(t,z) y^{\beta}
e^{i\Phi_{k}(t,y,z)/\varepsilon}
%
\\
 &=
i\tilde{g}(t,y,z,\varepsilon) e^{i\Phi_{k}(t,y,z)/\varepsilon}
   \sum_{j=0}^{\lceil \frac{k}{2} \rceil -1}\sum_{|\beta_1|=0}^{k-2j-1} 
\sum_{|\beta_2|=k+2}^{k+p+1} 
\frac{\varepsilon^{j-1}}{\beta_1!\beta_2!}a_{j,\beta_1}(t,z)
\phi_{\beta_2}(t,z) y^{\beta_1+\beta_2}.
%
\end{align*}
As before, this is a finite sum, now with terms of the form
\be{gtermform}
i\tilde{g}(t,y,z,\varepsilon)
\frac{\varepsilon^{j-1}}{\beta_1!\beta_2!}a_{j,\beta_1}(t,z)
\phi_{\beta_2}(t,z) y^{\beta_1+\beta_2}e^{i\Phi_{k}(t,y,z)/\varepsilon}.
\ee
It is again easy to check that $j-1+|\beta_1+\beta_2|/2\geq k/2$ for all terms.
There are therefore functions $g_{j}$,
multi-indices $\alpha_j$
and powers $\ell_j\geq 0$
such that for
some finite $N_q$,
$$
 A_k\left(e^{i\Phi_{k+p}/\varepsilon}
 -e^{i\Phi_{k}/\varepsilon}\right)
 = \varepsilon^{\frac{k}{2}}\sum_{j=0}^{N_q}\varepsilon^{\ell_j-|\alpha_j|/2}g_j(t,y,z,\varepsilon)(y-x(t,z))^{\alpha_j}e^{i\Phi_{k}(t,y-x(t,z),z)/\varepsilon},
$$
where $|\alpha_j| = k-2j+2$ 
if $\ell_j=0$, 
so, again, 
$|\alpha_j|$ then has the same parity as $k$.
Hence, the second term in \eq{vsplit}
contributes to $u_{k+p}-u_k$ as
\be{Phiops}
    \left(\frac{1}{2\pi\varepsilon}\right)^\frac{n}{2} 
   \int_{K_0} 
    A_k\left(e^{i\Phi_{k+p}/\varepsilon}
 -e^{i\Phi_{k}/\varepsilon}\right)
 \varrho_\eta dz
 = 
\varepsilon^{\frac{k}{2}} \sum_{j=0}^{N_q}\varepsilon^{\ell_j} \Int^{\alpha_j}
_{\Phi_k,x,g_j},
\ee
where, as before,
$\Phi_{k}$ and $x$ have properties (P1)--(P4) and (P6).

We have left to prove that $\Phi_k$, $x$, and $g_j$ have properties 
(P7) and (P8). 
By \eq{gtermform}, \eq{coeffsmooth} and \eq{acompact}, each $g_j$ is of the form
$f_j(t,z)\tilde{g}(t,y-x(t,z),z,\varepsilon)$
where $f_j(t,z)\in C^{\infty}({\mathcal K}_d)$ and supp$ f_j(t,\,\cdot\,)\subset K_0$ for $t\in[0,T]$. Hence,
$g_j(t,y,z,\varepsilon)\in C^{\infty}
([0,T]\times\Real^{n}\times{K}_d\times \Real^+)$, with compact support in $K_0$ for fixed $t,y,\varepsilon$.

To show \eq{B4} and \eq{B5},
%
%
%
we note first that
since both the phases $\Phi_k$, $\Phi_{k+p}$ satisfy (P$4'$), we have
for any $s\in [0,1]$, $(t,z)\in{\mathcal K}_d$, $y\in\Real^n$ and $\varepsilon>0$,
\begin{align}\lbeq{phimean}
  \left| e^{i[s\Phi_{k+p}(t,y,z)+(1-s)\Phi_k(t,y,z)]/\varepsilon}
  \varrho_\eta(y)  
  \right|
  &=
   e^{-s\Im\Phi_{k+p}(t,{y},z)-(1-s)\Im\Phi_k(t,{y},z)]/\varepsilon}\varrho_\eta(y)   \nonumber\\
&\leq   
   e^{-sw_{4,k+p}|{y}|^2/\varepsilon-(1-s)w_{4,k}|{y}|^2/\varepsilon}\nonumber\\
   &\leq e^{-\tilde{w}_4|y|^2},
   \end{align}
   where $w_{4,\ell}$ is the constant in (P$4'$) for $\Phi_\ell$ and
$\tilde{w}_4=\min(w_{4,k+p},w_{4,k})$.
To simplify the presentation in the remainder of the proof,  
we let $\tilde{y}=y_0-x(t,z)$ and drop the index $j$ from $g_j$ and $f_j$.
Then by \eq{phimean} and \eq{suplipfin},
\begin{align*}
\lefteqn{
\left|g(t,y_0,z,\varepsilon)e^{i\Phi_k(t,y_0-x(t,z),z)/\varepsilon}
\varrho_\eta(y_0-x(t,z))  
\right|}\hskip 20 mm &\\
    &=\left|f(t,z)\tilde{g}(t,\tilde{y},z,\varepsilon)e^{i\Phi_k(t,\tilde{y},z)/\varepsilon}
    \varrho_\eta(\tilde{y})  \right|\\
    &=\left|f(t,z)
    \int_0^1
   e^{i[s\Phi_{k+p}(t,\tilde{y},z)+(1-s)\Phi_k(t,\tilde{y},z)]/\varepsilon}     \varrho_\eta(\tilde{y})ds \right|
    \leq 
  C_1 e^{-{\tilde w}_4|\tilde{y}|^2/\varepsilon},
\end{align*}
for all $(t,z)\in {\mathcal K}_d$.
 This shows \eq{B4} and therefore (P7) with $D_7=C_1$ and $w_7=\tilde{w}_4$.
 
 Finally, for \eq{B5} we use the fact that 
 $\Phi_k(t,0,z)=\Phi_{k+p}(t,0,z)=\phi_0(t,z)$,
 which means that $\tilde{g}(t,0,z,\varepsilon)=1$.
We can therefore split
\begin{align*}
\lefteqn{
\Bigl(g(t,y_0,z,\varepsilon)-g(t,y_0,z_0,\varepsilon)\Bigr)
    e^{i\Phi_k(t,y_0-x(t,z),z)/\varepsilon}    \varrho_\eta(y_0-x(t,z))} \hskip 5 mm &\\
    &= 
\Bigl(f(t,z)\tilde{g}(t,\tilde{y},z,\varepsilon)-f(t,z_0)\tilde{g}(t,0,z_0,\varepsilon)\Bigr)
    e^{i\Phi_k(t,\tilde{y},z)/\varepsilon}    \varrho_\eta(\tilde{y}) \\
    &= 
f(t,z)\Bigl(\tilde{g}(t,\tilde{y},z,\varepsilon)-1\Bigr)
    e^{i\Phi_k(t,\tilde{y},z)/\varepsilon}    \varrho_\eta(\tilde{y})
    +
\Bigl(f(t,z)-f(t,z_0)\Bigr)
    e^{i\Phi_k(t,\tilde{y},z)/\varepsilon}    \varrho_\eta(\tilde{y}).
\end{align*}
Since $f$ is smooth, 
$t\in[0,T]$ and $z,z_0\in K_d$,
it follows from \eq{suplipfin}
and (P$4'$) that the second term can be estimated as
\be{fterm1}
\left|(f(t,z)-f(t,z_0))
    e^{i\Phi_k(t,\tilde{y},z)/\varepsilon}    \varrho_\eta(\tilde{y})\right|
    \leq C_2|z-z_0|e^{-w_{4,k}|\tilde{y}|^2/\varepsilon}.
\ee
For the first term we consider
\begin{align*}
\lefteqn{
\Bigl(\tilde{g}(t,\tilde{y},z,\varepsilon)-1\Bigr)
    e^{i\Phi_k(t,\tilde{y},z)/\varepsilon}
    }\hskip 1 cm
&    \\
&=    \int_0^1\left(
   e^{is(\Phi_{k+p}(t,\tilde{y},z)-\Phi_k(t,\tilde{y},z))/\varepsilon} -1
   \right)ds\times
e^{i\Phi_k(t,\tilde{y},z)/\varepsilon}
    \\
&=  \frac{i}{\varepsilon} (\Phi_{k+p}(t,\tilde{y},z)-\Phi_k(t,\tilde{y},z))
 \int_0^1\int_0^1s
   e^{i(sr\Phi_{k+p}(t,\tilde{y},z)+(1-sr)\Phi_k(t,\tilde{y},z))/\varepsilon}dsdr.\end{align*}
Hence, upon again using \eq{phimean}, \eq{coeffsmooth} and \eq{suplipfin},
\begin{align*}
\left|f(t,z)\Bigl(\tilde{g}(t,\tilde{y},z,\varepsilon)-1\Bigr)
    e^{i\Phi_k(t,\tilde{y},z)/\varepsilon}
    \varrho_\eta(\tilde{y})\right|
    &\leq
  \frac{C_1}{\varepsilon} |\Phi_{k+p}(t,\tilde{y},z)-\Phi_k(t,\tilde{y},z)|
  e^{-\tilde{w}_4|\tilde{y}|^2/\varepsilon}
  \\
  &\leq 
  \frac{C_1}{\varepsilon} 
  \sum_{|\beta|=k+2}^{k+p+1} \frac1{\beta!}|\phi_{\beta,\ell}(t,z)| |\tilde{y}|^{|\beta|}
  e^{-\tilde{w}_4 |\tilde{y}|^2/\varepsilon}\\
  &\leq 
  \frac{C_1'}{\varepsilon} 
  |\tilde{y}|^{k+2}
  e^{-\tilde{w}_4 |\tilde{y}|^2/\varepsilon}
\leq  \frac{C_3}{\varepsilon} 
|z-z_0|^{k+2}
  e^{-\tilde{w}_4 |\tilde{y}|^2/\varepsilon},
\end{align*}
where we also used the fact that by \eq{suplipfin},
$$
|\tilde{y}|=|x(t,z_0)-x(t,z)|\leq C |z-z_0|,
$$
whenever $t\in[0,T]$ and $z,z_0\in K_d$.
Together with \eq{fterm1} we 
thus get 
an estimate of the type \eq{B5} with
$D_8=\max(C_1,C_2,C_3)$, 
$w_8=\tilde{w}_4$,
$q=k+1$ and $\ell=1$, which satisfy $q\geq 2\ell$
as $k\geq 1$. 
This completes the proof of \theo{DiffRep}.


\subsection{Proof of \theo{Iestimate}}\lbsec{IestimateSec}


We henceforth consider a 
fixed $\delta>0$
and start by proving the 
two most simple cases in the theorem: when
$(t,y)$ is either outside
the essential support of the solution,
$(t,y)\in {\mathcal D}_{x,\delta}^c$, 
or close to a caustic point,
$(t,y)\in {\mathcal C}_{x,\delta}$.
We next consider the most difficult case, when
$(t,y)\in {\mathcal C}^c_{x,\delta}$. In particular,
showing the extra $\varepsilon^{1/2}$ factor when $|\alpha|$ is odd,
requires careful estimates.
To avoid breaking the flow of the arguments we move
most of the various lemmas'
proofs to \appen{proofs}. 

\subsubsection{Cases
$(t,y)\in {\mathcal D}_{x,\delta}^c$
and  
$(t,y)\in {\mathcal C}_{x,\delta}$}

For both these cases we make
use of the following integral estimate.
\begin{lemma}\lblem{intest1}
Let $U\subset \Real^n$ be a bounded measurable set.
Suppose $|y-x(t,z)|\geq a\geq 0$ when $z\in U$ for a fixed $t\in[0,T]$. If $b\geq 0$ and $c>0$ then
\be{invest}
  \int_{U}
|y-x(t,z)|^be^{-c|y-x(t,z)|^2/\varepsilon} dz \leq
C|U|\varepsilon^{b/2}e^{-ca^2/2\varepsilon},
\ee
where $C$ only depends on $b$ and $c$; it is
independent of $a$, $(t,y)\in[0,T]\times \Real^n$ and $\varepsilon>0$.
\end{lemma}
\begin{proof}
When $b=0$ the result is obviously true for $C=1$.
When $b>0$ we use the fact that $x^pe^{-x}\leq (p/e)^p$ for $p>0$ and $x\geq 0$.
Then
\begin{align*}
  \int_{U}
|y-x(t,z)|^be^{-c|y-x(t,z)|^2/\varepsilon} dz &\leq
  \int_{U}
|y-x(t,z)|^be^{-c|y-x(t,z)|^2/2\varepsilon}e^{-ca^2/2\varepsilon} dz
\nonumber
\\ &\leq 
\left(\frac{\varepsilon b}{c}\right)^{b/2}e^{-b/2}
e^{-ca^2/2\varepsilon}\int_Udz.
\end{align*}
This shows the lemma with $C=
\left(b/c\right)^{b/2}e^{-b/2}$.
\end{proof}

We now first suppose that $(t,y)\in {\mathcal D}_{x,\delta}^c$.
If $z\in K_0$,
then by definition
 \begin{align*}
 |y-x(t,z)| >\delta.
 \end{align*}
Therefore, by (P7) and \lem{intest1}, with $b=|\alpha|$, $c=w_7$ and $a=\delta$,
\begin{align*}
\left|\Int^{\alpha}_{\Phi,x,g}(t,y)\right|
 & \leq \varepsilon^{-\frac{n+|\alpha|}{2}}
\int_{K_0} \left|g(t,y,z,\varepsilon)(y-x(t,z))^\alpha e^{i\Phi(t,y-x(t,z),z)/\varepsilon}\varrho_\eta(y-x(t,z))\right|dz
\nonumber\\
& \leq
D_7 \varepsilon^{-\frac{n+|\alpha|}{2}}
\int_{K_0}
|y-x(t,z)|^{|\alpha|}
e^{-w_7|y-x(t,z)|^2/\varepsilon}dz
\nonumber\\
& \leq
D_7C |K_0|\varepsilon^{-\frac{n}{2}}
e^{-w_7\delta^2/2\varepsilon}
\leq C'e^{-{w}/\varepsilon},
\end{align*}
for ${w}<w_7\delta^2/2$, which proves the case $(t,y)\in {\mathcal D}_{x,\delta}^c$
since $D_7$ and $C$ are uniform constants in $t$ and $y$.

Second, suppose $(t,y)\in {\mathcal C}_{x,\delta}$.
Here, we simply use \lem{intest1} 
with $a=0$. This does not give an optimal estimate,
but slightly better than \eq{sobolevpessimistic}.
Hence, by (P7) and \lem{intest1} as above, with $b=|\alpha|$, $c=w_7$ and $a=0$,
\begin{align*}
\left|\Int^{\alpha}_{\Phi,x,g}(t,y)\right|
& \leq
D_7 \varepsilon^{-\frac{n+|\alpha|}{2}}
\int_{K_0}
|y-x(t,z)|^{|\alpha|}
e^{-w_7|y-x(t,z)|^2/\varepsilon}dz
\nonumber\\
& \leq
D_7C |K_0|\varepsilon^{-\frac{n}{2}}
\leq C' \varepsilon^{-\frac{n}{2}},
\end{align*}
where again $C'$ is independent of $(t,y)\in[0,T]\times\Real^n$.
This proves the theorem when $(t,y)\in {\mathcal C}_{x,\delta}$.


\subsubsection{Case $(t,y)\in {\mathcal C}^c_{x,\delta}$}

This is the most complicated case, in particular when $|\alpha|$ is odd.
The key idea of the proof is that
the ray function $x(t,z)$ is locally invertible in $z$
on the set ${\mathcal C}^c_{x,\delta}$. We derive this
property from a
uniform version of the inverse function theorem; see \theo{invfunc} below.
In order to carefully track the constants in the estimates,
and verify that they are independent of $(t,y)\in{\mathcal C}^c_{x,\delta}$,
we  define the following, finite, numbers
\be{R1def}
   R_1 = \mathop{\sup_{t\in [0,T]}}_{z\in{\rm conv}(K_d)} |J(t,z)|,
   \qquad
   R_2 = \sum_{j=1}^n\mathop{\sup_{t\in [0,T]}}_{z\in{\rm conv}(K_d)}  \left|D^2_zx_j(t,z)\right|,
\ee
where conv$(K)$ represents the convex hull of $K$ and $x=(x_1,\ldots,x_n)^T$.
This means that whenever $z,z'\in K_d$ and $t\in[0,T]$,
 \begin{align}
    |x(t,z)-x(t,z')|&\leq R_1|z-z'|,\lbeq{xdiff}\\
    |J(t,z)-J(t,z')|&\leq R_2|z-z'|,\lbeq{Jdiff}\\
    |x(t,z)-x(t,z')-J(t,z')(z-z')|&\leq \frac12 R_2|z-z'|^2.\lbeq{xdiff2}
 \end{align}
We also define the extended mapping $X:{\mathcal K}_d \mapsto [0,T]\times \Real^n$
as
$$
   X(t,z) = (t, x(t,z)),
$$
and we let ${\mathcal B}_r(z)$ be the open ball of radius $r$ centered at $z$.
We then have the following theorem for the ray function $x(t,z)$.

\begin{theorem}[Uniform Inverse Function Theorem]\lbtheo{invfunc}
Suppose $d'\in (0,d)$ and $\delta'>0$. 
Then there are numbers $R_{-1}$, $\rho>0$ and $0<r\leq d-d'$ such that, for all $(t,z_0)\in
{\mathcal K}_{d'}\setminus X^{-1}({\mathcal C}_{x,\delta'})$,
\begin{itemize}
\item $\bar{\mathcal B}_r(z_0)\subset K_d$,
\item 
$x(t,\,\cdot\,)$ restricted to ${\mathcal B}_r(z_0)$ is a diffeomorphism
on its image ${\mathcal V}_r(t,z_0):=x(t,{\mathcal B}_r(z_0))$,
\item 
${\mathcal V}_r(t,z_0)$ is open; if $y_0=x(t,z_0)$, then
$
  {\mathcal B}_{\rho}(y_0)\subset {\mathcal V}_r(t,z_0)$, and
\item the inverse of the Jacobian $J(t,z)$ is bounded on ${\mathcal B}_r(z_0)$,
$$
   \sup_{z\in {\mathcal B}_{r}(z_0)}|J^{-1}(t,z)|\leq R_{-1}.
$$
\end{itemize}
Note that $R_{-1}$, $r$ and $\rho$ are uniform in $(t,z_0)$ but in general
depend on $d'$ and $\delta'$. See \eq{Rm1def}, \eq{rdef} and \eq{rhodef} for their
precise definitions.
\end{theorem}

This result follows essentially in the same way as the
standard inverse function theorem. For completeness,
a proof is given in \appen{incfuncproof}. 

We let $\{z_j\}$ be the set of all solutions in $K_{d/2}$ to 
the equation $y=x(t,z)$.
Since $(t,y)\in {\mathcal C}_{x,\delta}^c \subset {\mathcal C}_{x,\delta/2}^c$
all points $(t,z_j)$ belong to ${\mathcal K}_{d/2}\setminus X^{-1}(
{\mathcal C}_{x,\delta/2})$. This set will be used extensively, and we
introduce the shorthand notation
$$
\bar{\mathcal K}:={\mathcal K}_{d/2}\setminus X^{-1}(
{\mathcal C}_{x,\delta/2}).
$$
We then apply \theo{invfunc}  with 
the parameters $d'=d/2$ and $\delta'=\delta/2$, and,
henceforth, we let $R_{-1}$, $r$ and $\rho$
be as given by the theorem with these parameters.
They then satisfy
\be{thisr}
0<r\leq d/2, \qquad R_{-1}, \rho>0.
\ee
We stress that the four bullet points in the theorem are then valid with
these numbers for {\em all}
$(t,z_0)\in \bar{\mathcal K}$.

In the remainder of the proof we will make use of
a few consequences of \theo{invfunc} 
which we collect in a corollary.

\begin{corollary}\label{uitcoroll}
The number of solutions $\{z_j\}$ in $K_{d/2}$ is bounded
by a number $M_\delta<\infty$, independently of $(t,y)\in {\mathcal C}_{x,\delta}^c$. The balls
$\{{\mathcal B}_{r/2}(z_j)\}$ are all disjoint.
Moreover, if 
$(t,z_0)\in \bar{\mathcal K}$ and
$x(t,z), x(t,z')\in {\mathcal B}_\rho(x(t,z_0))$, then
\be{zdiff}
|z-z'|\leq R_{-1}|x(t,z')-x(t,z)|.
\ee
\end{corollary}
\begin{proof}
If the number of solutions $\{z_j\}$ is more than one,
suppose
$|z_j-z_k|<r$ for some indices $j,k$. Then $z_j\in {\mathcal B}_r(z_k)$ 
and $x(t,z_j)=x(t,z_k)$
so $x(t,z)$ is not one-to-one on ${\mathcal B}_r(z_k)$. 
This contradicts the second point of \theo{invfunc}.
Hence, $|z_j-z_j|\geq r$ for all $j\neq k$
and the balls $\{{\mathcal B}_{r/2}(z_j)\}$ are disjoint.
Moreover, by the first point in \theo{invfunc},
each disjoint ball ${\mathcal B}_{r/2}(z_j)$ is a subset of
$K_d$ and
their total volume is therefore bounded by the volume of $K_d$.
The number of solutions must hence be finite, say $M$, and
$$
  |K_d|\geq \sum_{j=1}^M |{\mathcal B}_{r/2}(z_j)|= M\omega_n (r/2)^n
  \quad
  \Rightarrow\quad
  M\leq M_\delta =\frac{|K_d|2^n}{\omega_nr^n},
\qquad
\omega_n = \frac{\pi^{n/2}}{\Gamma(n/2+1)},
$$
where $\omega_n$ is the volume of the unit $n$-sphere. This shows the first
statement since
$M_\delta$ only depends on
$K_d$, $r$ and $n$.
For \eq{zdiff} we note that
by \theo{invfunc} there is a smooth inverse $m(t,x)$ satisfying
$m(t,x(t,z))=z$ for all $z\in {\mathcal B}_r(z_0)$. Let $y_0=x(t,z_0)$. Then
\begin{align*}
  |z-z'| &= |m(t,x(t,z))-m(t,x(t,z'))|
  \leq \sup_{y\in {\mathcal B}_\rho(y_0)}\left|D_x m(t,y)\right||x(t,z)-x(t,z')| \\
  &\leq  \sup_{q\in {\mathcal B}_r(z_0)}\left|J^{-1}(t,q)\right||x(t,z)-x(t,z')| 
  \leq R_{-1}|x(t,z)-x(t,z')|.
\end{align*}
For the last inequality we used the fourth point in \theo{invfunc}.
This shows the corollary.
\end{proof}

Hence, by Corollary \ref{uitcoroll}
the number of solutions $M$ to $y=x(t,z)$ in $K_{d/2}$ is finite.
We 
define the set $S\subset K_0$ as the points 
away from these solutions $\{z_j\}$,
$$
  S = 
  \begin{cases} K_0, & M=0,\\
  K_0\setminus \bigcup_{j=1}^{M} {\mathcal B}_{r/2}(z_j),& M\geq 0.
  \end{cases}
$$
Since 
$\{{\mathcal B}_{r/2}(z_j)\}$ are disjoint by
Corollary \ref{uitcoroll} we can then
%
split the integral as 
\begin{align*}
\Int^{\alpha}_{\Phi,x,g}(t,y) & = \varepsilon^{-\frac{n+|\alpha|}{2}}
\int_{K_0} g(t,y,z,\varepsilon)(y-x(t,z))^\alpha e^{i\Phi(t,y-x(t,z),z)/\varepsilon}\varrho_\eta(y-x(t,z))dz
\nonumber\\
& =
\int_{S}\cdots\ dz 
+\sum_{j=1}^M \int_{{\mathcal B}_{r/2}(z_j)\cap K_0} \cdots\ dz\\
&=
\int_{S}\cdots\ dz 
+\sum_{j=1}^M \int_{{\mathcal B}_{r/2}(z_j)} \cdots\ dz =: I_S + \sum_{j=1}^M I_{B_j}.
\end{align*}
Here we also used the fact from (P7)
that $g(t,y,\,\cdot\,\,,\varepsilon)$ is compactly supported in $K_0$.
We will show below that there are positive constants $w_s$, $C_S$ and $C_B$ that are independent
of $(t,y)\in
{\mathcal C}^c_{x,\delta}$ 
and $\varepsilon\in(0,1]$ such that
\be{isibest}
  |I_S|\leq C_S e^{-w_s/\varepsilon},\qquad
  |I_{B_j}|\leq C_B\begin{cases}
  1, & \text{$|\alpha|$ is even},\\
  \sqrt\varepsilon, &\text{$|\alpha|$ is odd}.
  \end{cases}
\ee
From Corollary \ref{uitcoroll} we have that
$M$ is bounded by $M_\delta$ uniformly in $(t,y)$. We therefore get the
desired estimate,
\begin{align*}
\left|\Int^{\alpha}_{\Phi,x,g}(t,y)\right|&\leq |I_S|
 + M_\delta \max_{j}|I_{B_j}|\leq 
 C_S e^{-w_s/\varepsilon}
+ M_\delta C_B
\begin{cases}
  1, & \text{$|\alpha|$ is even},\\
  \sqrt\varepsilon, &\text{$|\alpha|$ is odd},
  \end{cases}
  \\
  &\leq C\begin{cases}
  1, & \text{$|\alpha|$ is even},\\
  \sqrt\varepsilon, &\text{$|\alpha|$ is odd},
  \end{cases}
\end{align*}
for all $(t,y)\in
{\mathcal C}^c_{x,\delta}$ and $\varepsilon\in(0,1]$.


We now turn to proving \eq{isibest}. It will be done in three steps, one
for each case.

\medskip
\noindent{\bf Estimate of $I_S$.}\\
For this estimate we show that when $z\in S$ then 
$$|y-x(t,z)|\geq 
\bar{\rho} :=\min(\rho,r/2R_{-1},\delta/2).$$
Suppose first that $X(t,z)\not\in {\mathcal C}_{x,\delta/2}$. 
This implies that
$(t,z) \in \bar{\mathcal K}$ and \theo{invfunc} applies.
Assume $|y-x(t,z)|<\bar{\rho}\leq \rho$. 
Then $y\in {\mathcal B}_\rho(x(t,z))$
and by \theo{invfunc} there is a $z'\in {\mathcal B}_r(z)$ such that $y=x(t,z')$.
Since $z\in S\subset K_0$ and $r\leq d/2$
by \eq{thisr}, we have
$z'\in K_{d/2}$, so that $z'\in \{z_j\}$ and $M>0$. Hence,
%
%
by \eq{zdiff}, and the fact that $z\in S$,
$$
  \frac{r}{2}\leq |z-z'|\leq R_{-1}|x(t,z)-x(t,z')|<R_{-1}\bar{\rho}\leq \frac{r}{2},
$$
a contradiction. So $|y-x(t,z)|\geq \bar{\rho}$
if 
$X(t,z)\not\in {\mathcal C}_{x,\delta/2}$. 

Suppose instead that $X(t,z)\in{\mathcal C}_{x,\delta/2}$.
Then
$$
 |x(t,z)-y|
 ={\rm dist}(X(t,z),(t,y))
 \geq {\rm dist}((t,y),{\mathcal C}_x)-{\rm dist}(X(t,z),{\mathcal C}_x)
 \geq \delta -\delta/2=\delta/2\geq\bar{\rho},
$$
since $(t,y)\in {\mathcal C}_{x,\delta}^c$.
%
%
%
%
%
We have thus shown that if $z\in S$, then 
$|y-x(t,z)|\geq\bar{\rho}$. Therefore, by (P7)
and \lem{intest1}, with $C$ and $D_7$ independent of $(t,y)$ and $\varepsilon>0$,
\begin{align*}
|I_S| &=\left| \varepsilon^{-\frac{n+|\alpha|}{2}}
\int_{S} g(t,y,z,\varepsilon)(y-x(t,z))^\alpha e^{i\Phi(t,y-x(t,z),z)/\varepsilon}\varrho_\eta(y-x(t,z))dz
\right| \\
&\leq
D_7 \varepsilon^{-\frac{n+|\alpha|}{2}}
\int_{S} |y-x(t,z)|^{|\alpha|} e^{-w_7|y-x(t,z)|^2/\varepsilon}dz\\
&\leq
D_7C \varepsilon^{-\frac{n}{2}}|S|e^{-w_7\bar{\rho}^2/2\varepsilon}
\leq C_S e^{-w_s/\varepsilon},
\end{align*}
for $w_s<w_7\bar{\rho}^2/2$.
Here we also used the fact that $|S|\leq |K_0|<\infty$.
This shows the first inequality in \eq{isibest}.

\medskip
\noindent{\bf Estimate of $I_{B_j}$.}\\
The integrals $I_{B_j}$ are all of the form
\begin{align*}
I_B(t,z_0)=\varepsilon^{-\frac{n+|\alpha|}{2}}
\int_{{\mathcal B}_{\frac{r}{2}}(z_0)} g(t,y_0,z,\varepsilon)(y_0-x(t,z))^\alpha e^{i\Phi(t,y_0-x(t,z),z)/\varepsilon}\varrho_\eta(y_0-x(t,z))dz
\end{align*}
where 
$(t,z_0)\in \bar{\mathcal K}$,
$y_0=x(t,z_0)$ and the number
 $r$ is determined from \theo{invfunc}.
 It follows in particular that ${\mathcal B}_{r/2}(z_0)\subset K_d$ so
 that the estimates in properties (P$4'$), (P7) and (P8) can be used.
We now need to bound $I_{B}(t,z_0)$
with constants independent of $(t,z_0)\in\bar{\mathcal K}$
and $\varepsilon\in(0,1]$.
For this we use the following lemma. 
\begin{lemma}\lblem{intest2}
Suppose $r$ is given as above and
$y_0=x(t,z_0)$.
If $a,b\geq 0$ and $c>0$ there is a constant $C$ 
such that for all $(t,z_0)\in \bar{\mathcal K}$ and $\varepsilon>0$,
\be{intest2e2}
  \int_{{\mathcal B}_{\frac{r}{2}}(z_0)}
  |z-z_0|^a|y_0-x(t,z)|^be^{-c|y_0-x(t,z)|^2/\varepsilon} dz \leq C\varepsilon^{\frac{n+a+b}{2}}.
\ee
\end{lemma}
The proof is given in \appen{intest2proof}.

\medskip
\noindent{\bf Case when $|\alpha|$ even.}\\
For $|\alpha|$ even we directly apply (P7)
and \lem{intest2} to $I_B$
with $a=0$, $b=|\alpha|$ and $c=w_7$ to get
\begin{align*}
|I_B(t,z_0)|&\leq D_7\varepsilon^{-\frac{n+|\alpha|}{2}}
\int_{{\mathcal B}_{\frac{r}{2}}(z_0)} |y_0-x(t,z)|^{|\alpha|} e^{-w_7|y_0-x(t,z)|^2/\varepsilon}dz
\leq C'_B,
\end{align*}
for all $(t,z_0)\in \bar{\mathcal K}$ and $\varepsilon>0$.
This shows the first half of the second estimate in \eq{isibest}.

\medskip
\noindent{\bf Case when $|\alpha|$ odd.}\\
In this case we can gain
an additional factor 
of $\varepsilon^{1/2}$ if we make a careful estimate.
To do this, we approximate the phase $\Phi$ by
its leading order Taylor expansion in $z$ and show that the 
integral using the approximate $\Phi$
gives negligible contribution to the integral. The
following lemma details the phase approximation. It is proved
in \appen{PhiApproxproof}.

\begin{lemma}\lblem{PhiApprox}
Suppose $r$ is given as above and
$y_0=x(t,z_0)$.
%
%
If the phase $\Phi(t,y,z)$ and central ray $x(t,z)$ 
have properties (P1)--(P4) and (P6), then there is a bound $R_3$ such that
for all 
$(t,z_0)\in \bar{\mathcal K}$
and $z\in {\mathcal B}_{r/2}(z_0)$,
$$
  \left|\Phi(t,y_0-x(t,z),z) -
  \left(\Phi(t,0,z_0) + \frac12(z-z_0)^TA(t, z_0)(z-z_0)\right)\right| \leq R_3|z-z_0|^3,
$$
where $A(t,z_0)\in \Cnumbers^{n\times n}$.
The imaginary part of $A$ is
symmetric positive definite, and there exists $w_a>0$ such that
for all $(t,z_0)\in \bar{\mathcal K}$,
\be{Atilde}
\Im {A}(t,z_0)\geq w_a I.
\ee
\end{lemma}

We thus start by approximating $\Phi\approx \tilde{\Phi}$ and $I_B\approx\tilde{I}_B$ on ${\mathcal B}_{r/2}(z_0)$, where
$$
  \tilde\Phi(t,z,z_0) := \Phi(t,0,z_0)+\frac12 (z-z_0)^TA(t,z_0)(z-z_0),
$$ 
with $A(t,z_0)$ as in \lem{PhiApprox}, and 
$$
\tilde{I}_B(t,z_0):= 
\varepsilon^{-\frac{n+|\alpha|}{2}}
\int_{{\mathcal B}_{\frac{r}{2}}(z_0)} g(t,y_0,z_0,\varepsilon)(J(t,z_0)(z_0-z))^\alpha e^{i\tilde{\Phi}(t,z,z_0)/\varepsilon}\varrho_\eta(y_0-x(t,z))dz.
$$
We will now show that $\tilde{I}_B$ is exponentially small in $\varepsilon$.
To do this
we use the following lemma describing cancellations
occurring in integrals over odd mononomials multiplied by a Gaussian.
\begin{lemma}\lblem{ICancel}
Let $\alpha$ be an $n$-dimensional multi-index such that $|\alpha|$ is odd. 
For $A,R\in\Cnumbers^{n\times n}$ and any $r>0$,
$$
  \int_{{\mathcal B}_r(z_0)} (R(z-z_0))^\alpha e^{(z-z_0)^TA(z-z_0)}dz = 0.
$$
\end{lemma}
The proof of the lemma is given in \appen{ICancelproof}.
It shows that the $\tilde{I}_B$ integral, without $\varrho_\eta$, vanishes,
since
\begin{align*}
\lefteqn{\int_{{\mathcal B}_{\frac{r}{2}}(z_0)} g(t,y_0,z_0,\varepsilon)(J(t,z_0)(z_0-z))^\alpha e^{i\tilde{\Phi}(t,z,z_0)/\varepsilon}dz}\hskip 3 mm &\\
&=
e^{i\Phi(t,0,z_0)/\varepsilon}
g(t,y_0,z_0,\varepsilon)
\int_{{\mathcal B}_{\frac{r}{2}}(z_0)} 
(J(t,z_0)(z_0-z))^\alpha e^{\frac{i}{2}(z-z_0)^TA(t,z_0)(z-z_0)/\varepsilon}dz=0.
\end{align*}
Therefore,
\begin{align*}
\tilde{I}_B(t,z_0)&= \varepsilon^{-\frac{n+|\alpha|}{2}}
e^{i\Phi(t,0,z_0)/\varepsilon}
g(t,y_0,z_0,\varepsilon)\\
&\ \ \ \times
\int_{{\mathcal B}_{\frac{r}{2}}(z_0)} 
(J(t,z_0)(z_0-z))^\alpha e^{\frac{i}{2}(z-z_0)^TA(t,z_0)(z-z_0)/\varepsilon}
(\varrho_\eta(y_0-x(t,z))-1)
dz.
\end{align*}
Moreover, 
$\varrho_\eta(y-x)-1$ is identically zero for $|y-x|\leq \eta$,
and since $|y_0-x(t,z)|=|x(t,z_0)-x(t,z)|\leq R_1|z-z_0|$ when $z\in {\mathcal B}_r(z_0)$, 
we have 
by the positive definiteness of $\Im A$ given in \lem{PhiApprox},
\begin{align*}
\lefteqn{\left|\int_{{\mathcal B}_{\frac{r}{2}}(z_0)} 
(J(t,z_0)(z_0-z))^\alpha e^{\frac{i}{2}(z-z_0)^TA(t,z_0)(z-z_0)/\varepsilon}(\varrho_\eta(y_0-x(t,z))-1)dz\right|}\hskip 1 cm &\\
&\leq \int_{{\mathcal B}_{\frac{r}{2}}(z_0)} 
|J(t,z_0)|^{|\alpha|}|z_0-z|^{|\alpha|} e^{-\frac{1}{2}(z-z_0)^T\Im A(t,z_0)(z-z_0)/\varepsilon}|\varrho_\eta(y_0-x(t,z))-1|dz\\
&\leq 
\left(\frac{R_1r}{2}\right)^{|\alpha|}\int_{{\mathcal B}_{\frac{r}{2}}(z_0)} 
e^{-\frac{w_a}{2} |z-z_0|^2/\varepsilon}|\varrho_\eta(y_0-x(t,z))-1|dz\\
&\leq 
\left(\frac{R_1r}{2}\right)^{|\alpha|}
\left|{\mathcal B}_{\frac{r}{2}}(z_0)\right|
 e^{-w_a \eta^2/R_1^22\varepsilon}.
\end{align*}
%
%
%
Since $\Phi(t,0,z_0)$ is real by (P6), then by (P7),
noting that $y_0-x(t,z_0)=0$,
\be{gbounded}
  |g(t,y_0,z_0,\varepsilon)|=
  |g(t,y_0,z_0,\varepsilon)e^{i\Phi(t,0,z_0)/\varepsilon}
  \varrho_\eta(0)|\leq D_7,
\ee
where $D_7$ is clearly uniform in $(t,z_0)$.
Hence, there are constants
$\tilde{C}_B$ and $\tilde{w}$ such that
for all $(t,z_0)\in \bar{\mathcal K}$ and $\varepsilon>0$,
\be{Itest}
  |\tilde{I}_B(t,z_0)|\leq 
  \varepsilon^{-\frac{n+|\alpha|}{2}}D_7
  \left(\frac{R_1r}{2}\right)^{|\alpha|}
  \left|{\mathcal B}_{\frac{r}{2}}(z_0)\right|\, e^{-w_a \eta^2/R_1^22\varepsilon}
\leq
  \tilde{C}_Be^{-\tilde{w}/\varepsilon},
\ee
with $\tilde{w} < w_a\eta^2/2R_1^2$.

We next write the difference as
\begin{align*}
\varepsilon^{\frac{n+|\alpha|}{2}}(I_B-\tilde{I}_B)
   =
   \int_{{\mathcal B}_{\frac{r}{2}}(z_0)}
   (E_1+E_2+E_3)dz,
\end{align*}
where
\begin{align*}
E_1 &=
   [g(t,y_0,z,\varepsilon)-g(t,y_0,z_0,\varepsilon)]
   (y_0-x(t,z))^\alpha e^{i\Phi(t,y_0-x(t,z),z)/\varepsilon}\varrho_\eta(y_0-x(t,z)),\\
E_2 &=
   g(t,y_0,z_0,\varepsilon)
   [(y_0-x(t,z))^\alpha -(J(t,z_0)(z_0-z))^\alpha]e^{i\Phi(t,y_0-x(t,z),z)/\varepsilon}\varrho_\eta(y_0-x(t,z)),\\
E_3 &=
   g(t,y_0,z_0,\varepsilon)
   (J(t,z_0)(z_0-z))^\alpha\left[e^{i\Phi(t,y_0-x(t,z),z)/\varepsilon}
   -e^{i\tilde{\Phi}(t,z,z_0)/\varepsilon}\right]\varrho_\eta(y_0-x(t,z)).
\end{align*}
We will now consider these integrands in sequence. 

From (P8) it follows that
for all $(t,z_0)\in\bar{\mathcal K}$, $z\in {\mathcal B}_{r/2}(z_0)$ and $\varepsilon>0$,
\be{e1est}
  |E_1| \leq 
   D_8|z-z_0|\left(1+\frac{|z-z_0|^{q}}{\varepsilon^\ell}\right)
   |y_0-x(t,z)|^{|\alpha|}
   e^{-w_8|y_0-x(t,z)|^2/\varepsilon},
\ee
with $q\geq 2\ell$.

For $E_2$ we note first that
$$
 |a^\alpha-b^\alpha| = |(a-b+b)^\alpha-b^\alpha|=\left|
\sum_{\stackrel{\beta_1+\beta_2=\alpha}{\beta_2\neq\alpha}}\frac{\alpha!}{\beta_1!\beta_2!}
 (a-b)^{\beta_1} b^{\beta_2} \right|
\leq
\bar{C}(\alpha) \sum_{j=1}^{|\alpha|} |a-b|^{j}|b|^{|\alpha|-j}.
$$
Therefore, by using (P$4'$), \eq{gbounded}
and
\eq{xdiff2}
we get
for all $(t,z_0)\in\bar{\mathcal K}$, $z\in {\mathcal B}_{r/2}(z_0)$ and $\varepsilon>0$,
\begin{align}\lbeq{e2est}
|E_2|   &\leq \bar{C}(\alpha)D_7e^{-w_4|y_0-x(t,z)|^2/\varepsilon}
     \sum_{j=1}^{|\alpha|}
   |y_0-x(t,z)-J(t,z_0)(z_0-z)|^{j}|y_0-x(t,z)|^{|\alpha|-j}\nonumber\\
   &\leq
  \bar{C}(\alpha)   D_7
   \sum_{j=1}^{|\alpha|}\frac{R_2^j}{2^j}|z-z_0|^{2j}|y_0-x(t,z)|^{|\alpha|-j} 
   e^{-w_4|y_0-x(t,z)|^2/\varepsilon}\nonumber\\
   &\leq
C_2   \sum_{j=1}^{|\alpha|}|z-z_0|^{2j}|y_0-x(t,z)|^{|\alpha|-j} 
   e^{-w_4|y_0-x(t,z)|^2/\varepsilon},
\end{align}
where $C_2=\bar{C}(\alpha)D_7\max(R_2/2,(R_2/2)^{|\alpha|})$.

For $E_3$ we first need to approximate the phase difference factor when
$z\in {\mathcal B}_{r/2}(z_0)$ and $(t,z_0)\in\bar{\mathcal K}$.
By \lem{PhiApprox} and \eq{xdiff},
\begin{align*}
|\Phi-\tilde{\Phi}|&\leq R_3 |z-z_0|^3,\\
  \Im \tilde{\Phi} &= 
  \frac12 (z-z_0)^T\Im A(t,z_0)(z-z_0)\geq \frac{w_a|z-z_0|^2}{2} 
\geq  \frac{w_a|y_0-x(t,z)|^2}{2R_1^2} .
\end{align*}
Therefore, upon using (P$4'$),
\begin{align*}
\left|e^{i\Phi/\varepsilon}
   -e^{i\tilde{\Phi}/\varepsilon}\right|\varrho_\eta&=
   \left|\frac{i(\Phi-\tilde\Phi)}{\varepsilon}
   \int_0^1 e^{i(s\Phi+(1-s)\tilde\Phi)/\varepsilon}\varrho_\eta ds\right|
   \leq R_3
   \frac{|z-z_0|^3}{\varepsilon}
   e^{-\min(\Im\Phi,\Im\tilde\Phi)/\varepsilon}\\
   &\leq R_3
   \frac{|z-z_0|^3}{\varepsilon}
   e^{-\min(w_4,w_a/2R_1^2)|y_0-x(t,z)|^2/\varepsilon}.
\end{align*}
Then from \eq{gbounded}, with $w'=\min(w_4,w_a/2R_1^2)$
and $C_3 = R_3D_{7}R_1^{|\alpha|}$,
\be{e3est}
|E_3|
 \leq \frac{C_3}{\varepsilon}
   |z-z_0|^{|\alpha|+3}
   e^{-w'|y_0-x(t,z)|^2/\varepsilon},
\ee
for all $(t,z_0)\in\bar{\mathcal K}$, $z\in {\mathcal B}_{r/2}(z_0)$ and $\varepsilon>0$.
We note that all the $E_j$ terms can be bounded by a form that can
be estimated by \lem{intest2}. Indeed, if we define
$$
   U^\varepsilon(a,b):=
   |z-z_0|^{a}|y_0-x(t,z)|^{b} e^{-c|y_0-x(t,z)|^2/\varepsilon},\qquad
   c=\min(w_8,w'),
$$
and set $C_e=\max(D_8,C_2,C_3)$,
we can summarize \eq{e1est}, \eq{e2est}, \eq{e3est} as
\begin{align*}
\lefteqn{\varepsilon^{\frac{n+|\alpha|}{2}}|I_B(t,z_0)-\tilde{I}_B(t,z_0)| }\hskip 
3 mm & \\
&\leq C_e
   \int_{{\mathcal B}_{\frac{r}{2}}(z_0)} U^\varepsilon(1,|\alpha|)+
   \frac{1}{\varepsilon^\ell}U^\varepsilon(q+1,|\alpha|)
   +
   \sum_{j=1}^{|\alpha|}U^\varepsilon(2j,|\alpha|-j)
   +\frac{1}{\varepsilon}U^\varepsilon(|\alpha|+3,0)dz.
   \end{align*}
We then use \lem{intest2}, the constant in which we denote $C_L$.
We get for $0<\varepsilon\leq 1$,
\begin{align*}
\lefteqn{\varepsilon^{\frac{n+|\alpha|}{2}}|I_B(t,z_0)-\tilde{I}_B(t,z_0)| }
\hskip 3 mm &\\
 &\leq C_eC_L
   \left(\varepsilon^{\frac{n+1+|\alpha|}{2}}+
   \varepsilon^{\frac{n+q+1-2\ell+|\alpha|}{2}}
    +
   \sum_{j=1}^{|\alpha|}\varepsilon^{\frac{n+2j+|\alpha|-j}{2}}
   +\varepsilon^{\frac{n+|\alpha|+3+0-2}{2}}\right)
   \leq 
   C'\varepsilon^{\frac{n+1+|\alpha|}{2}},
   \end{align*}
   since $q\geq 2\ell$.
Together with \eq{Itest} we finally obtain
$$
   |I_B(t,z_0)| \leq |I_B(t,z_0)-\tilde{I}_B(t,z_0)|+|\tilde{I}_B(t,z_0)|\leq
   C'\sqrt{\varepsilon} +\tilde{C}_Be^{-\tilde{w}/\varepsilon}\leq C''_B\sqrt\varepsilon,
$$
for all $(t,z_0)\in\bar{\mathcal K}$ and $0<\varepsilon\leq 1$.
This shows the last part of the second inequality in \eq{isibest},
and completes the proof with $C_B=\max(C'_B,C''_B)$.


\bibliographystyle{plain}
\bibliography{new}


\appendix

\section{Proofs}\lbsec{proofs}


\subsection{Proof of \theo{invfunc}}\lbsec{incfuncproof}

The proof essentially follows the standard steps for proving the inverse
function theorem; see for instance \cite{rudin:76}.
We let ${\mathcal K}' = {\mathcal K}_{d'}\setminus X^{-1}({\mathcal C}_{x,\delta'})$
and consider the function
$$
\phi(z) = z + J^{-1}(t,z_0)(y -x(t,z)),
$$
with $(t,z_0)\in{\mathcal K}'$ and $y\in \Real^n$ fixed.
Since $J$ is non-singular on ${\mathcal K}'$, finding a fixed point $\phi(z)=z$ is equivalent to finding a solution to the equation $y=x(t,z)$.
We note that 
$J$ is non-singular also on the slightly larger set 
${\mathcal K}''={\mathcal K}_d\setminus X^{-1}({\mathcal C}_{x,\delta'/2})\supset {\mathcal K}'$
and we let $R_{-1}$ be an upper bound of $J^{-1}$
on this (compact) set,
\be{Rm1def}
   R_{-1} = \sup_{(t,z)\in {\mathcal K}''} |J^{-1}(t,z)|<\infty.
\ee
%
%
We then choose $r$ as
\be{rdef}
  r = \min\left(d-d', \frac{1}{2R_{-1}R_2},\frac{\delta'}{2R_1}\right)>0.
\ee
We note that if
$z\in \bar{\mathcal B}_r(z_0)$ we have
$$
   {\rm dist}(z,K_0)\leq |z-z_0| + {\rm dist}(z_0,K_0)\leq r + d'\leq d,
$$
Hence, $\bar{\mathcal B}_r(z_0)\subset K_d$ and
for $z_1,z_2\in \bar{\mathcal B}_{r}(z_0)$, using \eq{Jdiff},
\begin{align}\lbeq{philip}
|\phi(z_1)-\phi(z_2)| &\leq 
\max_{z\in \bar{\mathcal B}_{r}(z_0)}\left|D\phi(z)\right||z_1-z_2|=
\max_{z\in \bar{\mathcal B}_{r}(z_0)}\left|
I - J^{-1}(t,z_0) J(t,z)
\right||z_1-z_2|\nonumber\\
&\leq 
R_{-1}
\max_{z\in \bar{\mathcal B}_{r}(z_0)}\left|
J(t,z_0) - J(t,z)
\right||z_1-z_2|
\leq
R_{-1}R_2|z_1-z_2||z-z_0| 
\nonumber\\
&\leq 
R_{-1}R_2 r |z_1-z_2|
\leq \frac12 |z_1-z_2|.
\end{align}
If $z_1$ and $z_2$ are both, different, fixed points of $\phi$ we get an
impossible inequality.
It follows that $\phi$ has at most one fixed point in
$\bar{\mathcal B}_{r}(z_0)$ and therefore $x(t,z)$ is one-to-one on
$\bar{\mathcal B}_{r}(z_0)$.
We next show that ${\mathcal V}_r(t,z_0)$ is open.
For each
$y'\in {\mathcal V}_r(t,z_0)$ there is a $z'\in {\mathcal B}_{r}(z_0)$
and a $\lambda>0$, such that $y'=x(t,z')$ and
${B}_{\lambda}(z')\subset {\mathcal B}_{r}(z_0)$.
Let $\lambda'=\lambda/2R_{-1}$.
Then if $y\in {\mathcal B}_{\lambda'}(y')$,
\begin{align*}
  |\phi(z') -z'|  &= \left|
  J^{-1}(t,z_0)(y -y')\right|\leq R_{-1}|y-y'|< R_{-1}\lambda'=\frac12 \lambda.
\end{align*}
Consequentially, by \eq{philip}, if $z\in\bar{\mathcal B}_{\lambda}(z')
\subset \bar{\mathcal B}_{r}(z_0)$,
$$
|\phi(z) -z'|  \leq 
|\phi(z) -\phi(z')|+|\phi(z') -z'|< \frac12|z-z'|+ \frac12\lambda<\lambda.
$$
Hence, $\phi(z)\in \bar{\mathcal B}_\lambda(z')$ and $\phi$ is a contraction mapping
on $\bar{\mathcal B}_\lambda(z')$. This means that $\phi$ has a unique fixed point $z_*\in \bar{\mathcal B}_\lambda(z')$
at which
$y=x(t,z_*)$.
Thus $y\in {\mathcal V}_r(t,z_0)$, showing
that ${\mathcal B}_{\lambda'}(y')\subset {\mathcal V}_r(t,z_0)$. Hence,
${\mathcal V}_r(t,z_0)$ is open. 
In particular, if $y'=y_0=x(t,z_0)$
we can take
$\lambda = r$ and ${\mathcal B}_\rho(y_0)\subset {\mathcal V}_r(t,z_0)$
with 
\be{rhodef}
\rho = r/2R_{-1}.
\ee
For $z\in {\mathcal B}_r(z_0)$,
\begin{align*}
  {\rm dist}\Bigl((t,x(t,z)),{\mathcal C}_x\Bigr)&\geq 
  {\rm dist}\Bigl((t,x(t,z_0)),{\mathcal C}_x\Bigr)-
  {\rm dist}\Bigl((t,x(t,z)),(t,x(t,z_0))\Bigr)\\
  &=  {\rm dist}\Bigl((t,x(t,z_0)),{\mathcal C}_x\Bigr)-
  |x(t,z)-x(t,z_0)|\geq \delta' - R_1|z-z_0|\\
  &\geq  \delta' - R_1r\geq \delta' - \frac{\delta'}{2}=\frac{\delta'}{2},
\end{align*}
which shows that 
$(t, {\mathcal V}_r(t,z_0))\subset {\mathcal C}_{x,\delta'/2}^c$.
This means that $J(t,z)$ is invertible 
and $(t,z)\in {\mathcal K}''$
for all $z\in {\mathcal B}_r(z_0)$. The last point in the theorem then follows from \eq{Rm1def}.
That the inverse of $x(t,z)$ on ${\mathcal B}_r(z_0)$ is 
differentiable is proved in the same way as in \cite{rudin:76}.

\subsection{Proof of \lem{intest2}}\lbsec{intest2proof}
By \theo{invfunc} 
there is a smooth inverse of $x(t,\,\cdot\,)$ on ${\mathcal V}_r$.
Let $m(t,\,\cdot\,)$ be this inverse and $\rho$
the number paired with $r$ in \eq{thisr}.
Set $\tilde{\mathcal B}=m(t,{\mathcal B}_\rho(y_0))$.
We then split the integral as
\begin{align*}
\int_{{\mathcal B}_{\frac{r}{2}}(z_0)}\ldots dz=
\int_{{\mathcal B}_{\frac{r}{2}}(z_0)\setminus\tilde{\mathcal B}}\ldots dz+
\int_{{\mathcal B}_{\frac{r}{2}}(z_0)\cap \tilde{\mathcal B}}\ldots dz =: I_1+I_2.
\end{align*}
By construction we have $|y_0-x(t,z)|\geq \rho$
for  $z\in {\mathcal B}_{\frac{r}{2}}(z_0)\setminus\tilde{\mathcal B}$.
Therefore, by \lem{intest1},
\begin{align*}
|I_1|&\leq 
\left(\frac{r}{2}\right)^a\int_{{{\mathcal B}_{\frac{r}{2}}(z_0)\setminus\tilde{\mathcal B}}}|y_0-x(t,z)|^{b} e^{-c|y_0-x(t,z)|^2/\varepsilon}dz\\
&\leq C(b,c)\left(\frac{r}{2}\right)^a\left|{\mathcal B}_{\frac{r}{2}}(z_0)\setminus\tilde{\mathcal B}\right|\varepsilon^{b/2}e^{-c\rho^2/2\varepsilon}
\leq C'(a,b,c,n,r,\rho)\varepsilon^{\frac{n+a+b}{2}},
\end{align*}
for all $(t,z_0)\in\bar{\mathcal{K}}$ and $\varepsilon>0$.
Furthermore, on $\tilde{\mathcal B}$ we can use \eq{zdiff}, and
upon changing variables $y=x(t,z)$, we get
\begin{align}\lbeq{finalint}
|I_2|
&\leq R_{-1}^a
\int_{\tilde{\mathcal B}}|y_0-x(t,z)|^{a+b} e^{-c|y_0-x(t,z)|^2/\varepsilon}dz
\nonumber \\
&=R_{-1}^a
\int_{{\mathcal B}_\rho(y_0)}|y_0-y|^{a+b} e^{-c|y_0-y|^2/\varepsilon}|\det D_ym(t,y)|dy
\nonumber\\
&\leq 
R_{-1}^a
\sup_{y\in {\mathcal B}_\rho(y_0)} |\det D_ym(t,y)|
\int_{\Real^n} |y|^{a+b} e^{-c|y|^2/\varepsilon} dy
\nonumber\\
&=
R_{-1}^a
\sup_{y\in {\mathcal B}_\rho(y_0)} |\det D_ym(t,y)|
\varepsilon^{\frac{n+a+b}{2}}
\int_{\Real^n} |y|^{a+b} e^{-c|y|^2} dy.
\end{align}
For the determinant let $\lambda_j$ be the eigenvalues of $A\in\Real^{n\times n}$. Then
$$
   |\det A| = \prod |\lambda_j| \leq |\lambda_{\rm max}|^n
   = |A^TA|_2^{n/2}\leq |A|^n_2.
$$
Hence, by the fourth point in \theo{invfunc},
$$
\sup_{y\in {\mathcal B}_\rho(y_0)}
\left|\det D_ym(t,y)\right|
\leq \sup_{y\in {\mathcal B}_\rho(y_0)}
\left|D_y m(t,y)\right|^n= \sup_{z\in \tilde{\mathcal B}}
\left|J^{-1}(t,z)\right|^n\leq R_{-1}^n.
$$
Finally,
$$
|I_2|\leq R_{-1}^{a+n} C''(a,b,c,n)\varepsilon^{\frac{n+a+b}{2}},
$$
where $C''(a,b,c,n)$ is the value of the last integral in \eq{finalint}.
The result follows with $C=\max(C',R_{-1}^{a+n}C'')$,
since all these constants are uniform in 
$(t,z_0)\in\bar{\mathcal{K}}$.

\subsection{Proof of \lem{PhiApprox}}\lbsec{PhiApproxproof}

We consider $(t,z_0)\in \bar{\mathcal K}$.
By \theo{invfunc}, we have ${\mathcal B}_{r/2}(z_0)\subset K_d$
for these $(t,z_0)$.
For simplicity we henceforth drop the $t$-dependence in the notation.
By (P1) and (P2) we can Taylor expand $\Phi(x(z_0)-x(z),z)$
around $z=z_0$, and since $K_d$ is compact, we can
bound the remainder term using a constant $R_3$ that is uniform in
$(t,z_0)\in \bar{\mathcal K}$ and $z\in {\mathcal B}_{r/2}(z_0)$,
\begin{align*}
   \Biggl|\Phi(y_0-x(z),z) -&\Bigl(
   \Phi(0,z_0)- \left[J(z_0)^T\nabla_y\Phi(0,z_0)-\nabla_z\Phi(0,z_0)\right]\cdot(z-z_0)\\
  & +\frac12 (z-z_0)\cdot \left.D^2_z [\Phi(x(z_0)-x(z),z)]\right|_{z=z_0} (z-z_0)\Bigr)\Biggr| \leq R_3|z-z_0|^3.
\end{align*}
Using also (P6)
we get
\begin{align*}
   \Biggl|\Phi(y_0-x(z),z) -\Bigl(
   \Phi(0,z_0)+\frac12 (z-z_0)\cdot A(z_0)(z-z_0)\Bigr)\Biggr| \leq R_3|z-z_0|^3,
\end{align*}
where
\begin{align*}
  A(z_0) &= \left.D^2_z [\Phi(x(z_0)-x(z),z)]\right|_{z=z_0} =
  J(z_0)^TD^2_y\Phi(0,z_0)J(z_0) -J(z_0)D^2_{yz}\Phi(0,z_0)\\
  &\ \ \ \left.
  D_z\left(-J(z)^T\nabla_y\Phi(0,z)+\nabla_z\Phi(0,z)\right)\right|_{z=z_0}\\
  &= J(z_0)^TD^2_y\Phi(0,z_0)J(z_0) -J(z_0)D^2_{yz}\Phi(0,z_0).
\end{align*}
We have left to show the properties of $A(z_0)$. 
Since $\nabla_y\Phi(0,z_0)$ is real by (P6), so is $D^2_{yz}\Phi(0,z_0)$.
Clearly $J(z_0)$ is also real. Hence, 
$$
  \Im A(z_0) = J(z_0)^T\left(\Im D^2_y\Phi(0,z_0)\right)J(z_0),
$$
which is symmetric. To show the positive definiteness, we note that by (P6) 
both $\Phi(0,z_0)$ and $\nabla_y\Phi(0,z_0)$ are real and therefore,
$$
   \frac12y^T\Im D^2_y\Phi(0,z_0)y = \Im\Phi(y,z_0) + O(|y|^3).
$$
Moreover, for $|y|\leq 2\eta$ we have from (P4) that $\Im\Phi(y,z_0)\geq w_4 |y|^2$, so
$$
   \frac12y^T\Im D^2_y\Phi(0,z_0)y \geq w_4 |y|^2 + O(|y|^3).
$$
Setting $y=sv$ for some arbitrary $v\in \Real^n$ and scalar $s>0$, we 
therefore get
$$
   \frac12v^T\Im D^2_y\Phi(0,z_0)v =
  \frac{1}{2s^2}  (sv)^T\Im D^2_y\Phi(0,z_0)(sv) 
   \geq w_4 |v|^2 + O(s|v|^3),
$$
when $s$ is sufficiently small. Letting $s\to 0$ shows that
$\Im D^2_y\Phi(0,z_0)\geq 2w_4$.
Thus, finally,
$$
  v^T\Im A(z_0)v = (J(z_0) v)^T \Im D^2_y\Phi(0,z_0) (J(z_0)v)
  \geq 2w_4 |J(z_0)v|^2\geq \frac{2w_4}{R_{-1}^2}|v|^2,
$$
since $|v|=|J^{-1}(z_0)J(z_0)v|\leq R_{-1}|J(z_0)v|$
by \theo{invfunc}.
This concludes the proof with $w_a=2w_4/R_{-1}^2$.

\subsection{Proof of \lem{ICancel}}\lbsec{ICancelproof}

Without loss of generality we can take $z_0=0$.
By symmetry ${\mathcal B}_r(0)$ is invariant under the transformation $z\to -z$, so
$$
  \int_{{\mathcal B}_r(0)} (Rz)^\alpha e^{z^TAz}dz =
\int_{{\mathcal B}_r(0)} (R(-z))^\alpha e^{z^TAz}dz=
\frac12\int_{{\mathcal B}_r(0)} \left((Rz)^\alpha+(R(-z))^\alpha\right) e^{z^TAz}dz.
$$
Moreover, $(Rz)^{\alpha}$ will be of the form
$$
(Rz)^{\alpha} = \sum c_j z^{\ell_j},\qquad |\ell_j|=|{\alpha}|,
$$
for some 
multi-indices $\ell_j$ and
constants $c_j$, determined by the elements of $R$.
Hence,
\begin{align*}
  \int_{{\mathcal B}_r(0)} (Rz)^\alpha e^{z^TAz}dz &=
\frac12\sum c_j\int_{{\mathcal B}_r(0)} 
(z^{\ell_j}+(-z)^{\ell_j}) e^{z^TAz}dz\\
&=
\frac12\sum c_j\int_{{\mathcal B}_r(0)} 
z^{\ell_j}(1+(-1)^{|\ell_j|}) e^{z^TAz}dz = 0,
\end{align*}
if $|\ell_j|=|\alpha|$ is odd.

\end{document}